\documentclass[11pt,leqno]{article}
\usepackage[margin=1in]{geometry} 
\usepackage{amssymb,amsfonts,amsmath,bbm,mathrsfs,stmaryrd,mathtools}
\usepackage{xcolor}
\usepackage{url}

\usepackage{extarrows}

\usepackage[shortlabels]{enumitem}
\usepackage{tensor}

\usepackage{xr}
\usepackage[T1]{fontenc}

\usepackage[colorlinks,
linkcolor=black!75!red,
citecolor=blue,
pdftitle={},
pdfproducer={pdfLaTeX},
pdfpagemode=None,
bookmarksopen=true,
bookmarksnumbered=true,
backref=page]{hyperref}

\usepackage{tikz}
\usetikzlibrary{arrows,calc,decorations.pathreplacing,decorations.markings,intersections,shapes.geometric,through,fit,shapes.symbols,positioning,decorations.pathmorphing}

\usepackage{braket}

\usepackage[amsmath,thmmarks,hyperref]{ntheorem}
\usepackage{cleveref}

\newcommand\nthalias[1]{\AddToHook{env/#1/begin}{\crefalias{lemma}{#1}}}

\nthalias{definition}
\nthalias{example}
\nthalias{examples}
\nthalias{remark}
\nthalias{remarks}
\nthalias{convention}
\nthalias{notation}
\nthalias{construction}
\nthalias{theoremN}
\nthalias{propositionN}
\nthalias{corollaryN}
\nthalias{lemma}
\nthalias{proposition}
\nthalias{corollary}
\nthalias{theorem}
\nthalias{conjecture}
\nthalias{question}
\nthalias{assumption}

\creflabelformat{enumi}{#2#1#3}

\crefname{section}{Section}{Sections}
\crefformat{section}{#2Section~#1#3} 
\Crefformat{section}{#2Section~#1#3} 

\crefname{subsection}{\S}{\S\S}
\AtBeginDocument{%
  \crefformat{subsection}{#2\S#1#3}%
  \Crefformat{subsection}{#2\S#1#3}%
}

\crefname{subsubsection}{\S}{\S\S}
\AtBeginDocument{%
  \crefformat{subsubsection}{#2\S#1#3}%
  \Crefformat{subsubsection}{#2\S#1#3}%
}

\theoremstyle{plain}

\newtheorem{lemma}{Lemma}[section]
\newtheorem{proposition}[lemma]{Proposition}
\newtheorem{corollary}[lemma]{Corollary}
\newtheorem{theorem}[lemma]{Theorem}

\theoremstyle{plain}
\theoremnumbering{Alph}
\newtheorem{theoremN}{Theorem}

\theoremstyle{plain}
\theorembodyfont{\upshape}
\theoremsymbol{\ensuremath{\blacklozenge}}

\newtheorem{definition}[lemma]{Definition}
\newtheorem{example}[lemma]{Example}

\newtheorem{remark}[lemma]{Remark}
\newtheorem{remarks}[lemma]{Remarks}

\crefname{definition}{definition}{definitions}
\crefformat{definition}{#2definition~#1#3} 
\Crefformat{definition}{#2Definition~#1#3} 

\crefname{ex}{example}{examples}
\crefformat{example}{#2example~#1#3} 
\Crefformat{example}{#2Example~#1#3} 

\crefname{exs}{example}{examples}
\crefformat{examples}{#2example~#1#3} 
\Crefformat{examples}{#2Example~#1#3} 

\crefname{remark}{remark}{remarks}
\crefformat{remark}{#2remark~#1#3} 
\Crefformat{remark}{#2Remark~#1#3} 

\crefname{remarks}{remark}{remarks}
\crefformat{remarks}{#2remark~#1#3} 
\Crefformat{remarks}{#2Remark~#1#3} 

\crefname{convention}{convention}{conventions}
\crefformat{convention}{#2convention~#1#3} 
\Crefformat{convention}{#2Convention~#1#3} 

\crefname{notation}{notation}{notations}
\crefformat{notation}{#2notation~#1#3} 
\Crefformat{notation}{#2Notation~#1#3} 

\crefname{table}{table}{tables}
\crefformat{table}{#2table~#1#3} 
\Crefformat{table}{#2Table~#1#3}

\crefname{lemma}{lemma}{lemmas}
\crefformat{lemma}{#2lemma~#1#3} 
\Crefformat{lemma}{#2Lemma~#1#3} 

\crefname{proposition}{proposition}{propositions}
\crefformat{proposition}{#2proposition~#1#3} 
\Crefformat{proposition}{#2Proposition~#1#3} 

\crefname{propositionN}{proposition}{propositions}
\crefformat{propositionN}{#2proposition~#1#3} 
\Crefformat{propositionN}{#2Proposition~#1#3} 

\crefname{corollary}{corollary}{corollaries}
\crefformat{corollary}{#2corollary~#1#3} 
\Crefformat{corollary}{#2Corollary~#1#3} 

\crefname{corollaryN}{corollary}{corollaries}
\crefformat{corollaryN}{#2corollary~#1#3} 
\Crefformat{corollaryN}{#2Corollary~#1#3} 

\crefname{theorem}{theorem}{theorems}
\crefformat{theorem}{#2theorem~#1#3} 
\Crefformat{theorem}{#2Theorem~#1#3} 

\crefname{theoremN}{theorem}{theorems}
\crefformat{theoremN}{#2theorem~#1#3} 
\Crefformat{theoremN}{#2Theorem~#1#3} 

\crefname{enumi}{}{}
\crefformat{enumi}{#2#1#3}
\Crefformat{enumi}{#2#1#3}

\crefname{assumption}{assumption}{Assumptions}
\crefformat{assumption}{#2assumption~#1#3} 
\Crefformat{assumption}{#2Assumption~#1#3} 

\crefname{construction}{construction}{Constructions}
\crefformat{construction}{#2construction~#1#3} 
\Crefformat{construction}{#2Construction~#1#3} 

\crefname{equation}{}{}
\crefformat{equation}{(#2#1#3)} 
\Crefformat{equation}{(#2#1#3)}

\numberwithin{equation}{section}

\theoremstyle{nonumberplain}
\theoremsymbol{\ensuremath{\blacksquare}}

\newtheorem{proof}{Proof}
\newcommand\pf[1]{\newtheorem{#1}{Proof of \Cref{#1}}}

\newcommand\bZ{{\mathbb Z}}

\newcommand\cC{{\mathcal C}}

\newcommand\cM{{\mathcal M}}

\newcommand\cS{{\mathcal S}}

\DeclareMathOperator{\id}{id}
\DeclareMathOperator{\lcm}{lcm}
\DeclareMathOperator{\End}{\mathrm{End}}

\DeclareMathOperator{\spn}{\mathrm{span}}

\newcommand{\cat}[1]{\textsc{#1}}

\newcommand{\qedhere}{\mbox{}\hfill\ensuremath{\blacksquare}}

\newcommand{\xrightarrowdbl}[2][]{%
  \xrightarrow[#1]{#2}\mathrel{\mkern-14mu}\rightarrow
}

\title{Quantum group coproducts and universality under scalar extensions}
\author{Alexandru Chirvasitu}

\begin{document}

\date{}

\newcommand{\Addresses}{{
  \bigskip
  \footnotesize

  \textsc{Department of Mathematics, University at Buffalo}
  \par\nopagebreak
  \textsc{Buffalo, NY 14260-2900, USA}  
  \par\nopagebreak
  \textit{E-mail address}: \texttt{achirvas@buffalo.edu}


}}

\maketitle

\begin{abstract}
  We characterize the families of bialgebras or Hopf algebras over fields for which the product in the corresponding category is finite-dimensional, answering a question of M. Lorenz: if the ground field is infinite then bialgebra or Hopf products are finite-dimensional precisely when the factors are, with at most one of dimension $>1$; over finite fields the necessary and sufficient condition is instead that factors be finite-dimensional with at most finitely many of dimension $>1$; finally, these statements hold for coalgebras as well, provided the family is finite. We also characterize (a) finite field extensions as precisely those whose underlying scalar extension functor preserves coalgebra, or bialgebra, or Hopf algebra products (correcting an error in the literature); (b) algebraic field extensions as those along which finite coalgebra (bialgebra, Hopf algebra) products are preserved; and (c) again algebraic field extensions as precisely those which intertwine cofree coalgebras on vector spaces, or cofree bialgebras (Hopf algebras) on algebras. 
\end{abstract}

\noindent {\em Key words:
  Hopf algebra;
  Tannaka reconstruction;
  adjoint functor;
  affine algebra;
  algebraic extension;
  bialgebra;
  coalgebra;
  cofree;
  comonadic;
  coproduct;
  coradical;
  cosemisimple;
  finite dual;
  free;
  geometric quotient;
  locally presentable category;
  monadic;
  product;
  pseudo-pullback;
  sink
  
}

\vspace{.5cm}


\tableofcontents

\section*{Introduction}

The present paper is in part motivated by the natural question \cite{lor_fdhprod} of when and how often products $\displaystyle \prod^{\cat{HAlg}}_i H_i$ in the category $\tensor*[_{\Bbbk}]{\cat{HAlg}}{}$ of Hopf algebras over a field $\Bbbk$ are finite-dimensional. The quantum groups in the title are those attached to the Hopf algebras in question {\it contra}variantly, as common in non-commutative geometry (e.g. \cite[Chapters 7 and 13]{cp_qg}); it is in this sense that we are here concerned with {\it co}products: dual counterparts to Hopf (and more generally, bialgebra or coalgebra) products. 

That the categories $\tensor*[_{\Bbbk}]{\cat{HAlg}}{}$ (along with those of bialgebras and coalgebras, $\tensor*[_{\Bbbk}]{\cat{BiAlg}}{}$ and $\tensor*[_{\Bbbk}]{\cat{Coalg}}{}$) admit arbitrary (co)limits and afford a wealth of universal constructions ((co)free objects, etc.) is well familiar, and the literature on category-theoretic aspects of the subject abounds: \cite{agore_mono,agore_hopf,zbMATH05696924,MR4712418,zbMATH07854093,zbMATH05312006,zbMATH06696042,zbMATH06696043,zbMATH03344702}, \cite[\S\S 6.4, 12.2 and so on]{swe}, \cite[\S 1.6]{dnr}, etc. etc. All categories mentioned are in fact (e.g. by \cite[Lemmas 1 and 2 and Theorem 6]{zbMATH06696043})  {\it locally presentable} \cite[Definition 1.17]{ar} and are thus amenable to various adjoint functor theorems \cite[\S\S 0.7 and 1.66]{ar}. 

Devoted to variations on that initial motivating question, \Cref{se:smallprod} confirms (in \Cref{th:infdimprod}, \Cref{cor:finfld} and \Cref{th:coalgfdprod}) that it is only in the self-evident, trivial cases that a coalgebra (bialgebra, Hopf algebra) product can be small. The focus is specifically on the three categories $\cat{Coalg}$, $\cat{BiAlg}$ and $\cat{HAlg}$ for definiteness: the reader can easily supply analogues for Hopf algebras with bijective antipode and cousins (e.g. Hopf algebras with antipode $S$ satisfying $S^{2d}=\id$ for some fixed $d$). 

\begin{theoremN}\label{thn:fdprod}
  Let $\Bbbk$ be a field. 
  \begin{enumerate}[(1),wide]
  \item If $\Bbbk$ is infinite, a product $\prod_i H_i$ in the category of $\Bbbk$-bialgebras or Hopf $\Bbbk$-algebras is finite-dimensional if and only if all $H_i$ are, with at most one of dimension $>1$.

  \item If $\Bbbk$ is finite, such a product is infinite-dimensional all $H_i$ are, with only finitely many of dimension $>1$.

  \item All of the above applies to finite families of $\Bbbk$-coalgebras.  \qedhere
  \end{enumerate}
\end{theoremN}

In part prompted by the natural need to adjust the ground field in addressing \Cref{thn:fdprod}, and in part in order to correct an error (of the author's) in the literature, \Cref{se:fldext} turns to permanence properties for universal constructions (products, cofree objects) in categories of coalgebras or bialgebras under extending scalars via the functor $(-)_{\Bbbk'}:=-\otimes_{\Bbbk}\Bbbk'$ attached to a field extension $\Bbbk\le \Bbbk'$. 

The mistake alluded to above is the claim \cite[Theorem 4.1(2)]{MR4712418} that $(-)_{\Bbbk'}$ preserves arbitrary products of coalgebras, or bialgebras, or Hopf algebras: Examples \ref{ex:scalewrong} and \ref{ex:dualfields} each give instances of failure (or rather a families of them). Allowing for variations in taste, the corrected statement(s) may well be more interesting:
\begin{itemize}
\item that specific permanence property in fact characterizes the {\it finite} field extensions (\Cref{th:finextradj});

\item and furthermore, there is an analogue characterizing {\it algebraic} extensions by relaxing the product-preservation property appropriately (\Cref{th:algextradj}). 
\end{itemize}

\begin{theoremN}\label{thn:finextradj}
  \begin{enumerate}[(1),wide]
  \item A field extension $\Bbbk\le \Bbbk'$ is finite if and only if the corresponding scalar extension functor $(-)_{\Bbbk'}$
    \begin{itemize}[wide]
    \item is a right adjoint between the corresponding categories of coalgebras (equivalently, bialgebras or Hopf algebras);

    \item or preserves products in $\cat{Coalg}$ (again equivalently, in $\cat{BiAlg}$ or $\cat{HAlg}$).
    \end{itemize}

  \item Similarly, a field extension $\Bbbk\le \Bbbk'$ is finite if and only if the corresponding scalar extension functor $(-)_{\Bbbk'}$ preserves finite limits (equivalently, finite products) in any one of the categories $\cat{Coalg}$, $\cat{BiAlg}$ or $\cat{HAlg}$.  \qedhere
  \end{enumerate}  
\end{theoremN}

That discussion in turn generates a bit of downstream momentum: it is clear from the above statement that, in analyzing ways in which $(-)_{\Bbbk'}$ might fail to preserve various universal constructions, the distinction between algebraic and transcendental field extensions $\Bbbk\le \Bbbk'$ is crucial. The following characterization (\Cref{th:extprescf} below) of algebraic extensions is an upshot of that analysis, and in the same spirit as \Cref{thn:finextradj}. 

\begin{theoremN}\label{thn:extprescf}
  A field extension $\Bbbk\le \Bbbk'$ is algebraic if and only if the corresponding scalar extension functor $(-)_{\Bbbk'}$
  \begin{itemize}[wide]
  \item preserves cofree coalgebras on arbitrary vector spaces or equivalently, cofree bialgebras (or Hopf algebras) on arbitrary algebras;
  \item equivalently, any of the above for only the 1-dimensional vector spaces or algebras.  \qedhere
  \end{itemize}
\end{theoremN}

\subsection*{Acknowledgements}

I am grateful for insightful comments, questions and suggestions from A. Agore, M. Lorenz and G. Militaru. This work is partially supported by NSF grant DMS-2001128, and is part of the project Graph Algebras partially supported by EU grant HORIZON-MSCA- SE-2021 Project 101086394.


\section{Small coalgebra (bialgebra, Hopf algebra) products}\label{se:smallprod}

All (co)algebras are assumed (co)unital, along with their respective morphisms. Everything in sight is linear over a field $\Bbbk$, and we frequently denote scalar extensions along $\Bbbk\to \Bbbk'$ by $(-)_{\Bbbk'}$ (mostly fields, but the notation applies generally, for ring morphisms). For background on coalgebras, bialgebras and Hopf algebras we refer the reader to \cite{dnr,mont,rad,swe}; more specific citations are scattered throughout the text, where needed. 

As is customary in the category-theoretic literature (e.g. \cite[Definition 19.3]{ahs}), `$\top$' symbols occasionally indicate adjunctions by having the narrow end point towards the left adjoint. $\cC(-,-)$ stands for morphisms in the category $\cC$. Some common categories in use below include

\begin{itemize}[wide]
\item $\tensor*[_{\bullet}]{\cM}{}$ and $\cM_{\bullet}$, left and right modules respectively and similarly, $\cM^{\bullet}$ and $\tensor*[^{\bullet}]{\cM}{}$ for comodules;

\item additional `$f$' subscripts indicate finite-dimensional objects therein, as in $\cM^C_f$ for finite-dimensional right $C$-comodules;

\item $\cat{Vect}$, $\cat{Alg}$, $\cat{Coalg}$, $\cat{BiAlg}$ and $\cat{HAlg}$ denote categories of vector spaces, algebras, coalgebras, bialgebras and Hopf algebras respectively, occasionally decorated with the ground field for clarity or emphasis (e.g. $\tensor*[_{\Bbbk}]{\cat{Coalg}}{}$ for $\Bbbk$-coalgebras). 
\end{itemize}


We need some 2-categorical (or bicategorical) background, as the reader can find covered in \cite{fiore_2cat,jy_2dcat,zbMATH05659661} and numerous other sources. In particular, extend the notion of a {\it pseudo-pullback} of \cite[\S 6.10, Example 15]{zbMATH05659661} from pairs of functors to families of co-terminal functors (i.e. {\it sinks} \cite[Definition 10.62]{ahs}). 

\begin{definition}\label{def:plbk}
  A {\it pseudo-pullback} $\cat{pplb}(F_i, i\in I)$ is the {\it pseudo-limit} \cite[\S 6.10]{zbMATH05659661} of the diagram consisting of a family of functors $\cC_i\xrightarrow{F_i}\cC$ (a {\it sink} with {\it codomain} $\cC$ and {\it domain} $(\cC_i)_i$). Concretely, it is the category with
  \begin{itemize}[wide]
  \item objects consisting of tuples $(c_i\in \cC_i,\ i\in I\ ;\ c\in \cC)$ and $\cC$-isomorphisms $F_i c_i\xrightarrow[\cong]{\varphi_i}c$;

  \item and morphisms $(c_i;c;\varphi_i)\to (c'_i;c';\varphi'_i)$ consisting of $\cC_i$-morphisms $c_i\to c'_i$ and a $\cC$-morphism $c\to c'$ making the obvious diagrams (involving $\varphi_i$ and $\varphi'_i$) commute. 
  \end{itemize}
  $\cat{pplb}(F_i)$, in words, is the universal category equipped with functors to the $\cC_i$ and $\cC$ and natural isomorphisms
  \begin{equation*}
    \begin{tikzpicture}[>=stealth,auto,baseline=(current  bounding  box.center)]
      \path[anchor=base] 
      (0,0) node (l) {$\bullet$}
      +(2,.5) node (u) {$\cC_i$}
      +(4,0) node (r) {$\cC$:}
      +(2,.13) node () {\rotatebox[origin=c]{-90}{$\cong$}}
      ;
      \draw[->] (l) to[bend left=6] node[pos=.5,auto] {$\scriptstyle $} (u);
      \draw[->] (u) to[bend left=6] node[pos=.5,auto] {$\scriptstyle F_i$} (r);
      \draw[->] (l) to[bend right=6] node[pos=.5,auto,swap] {$\scriptstyle $} (r);
    \end{tikzpicture}
  \end{equation*}
  the top left-hand arrows are $(c_i;c;\varphi_i)\mapsto c_i$, the bottom arrow selects the $c$ instead, and the natural isomorphisms are what the $\varphi_i$ respectively aggregate to. 
\end{definition}

That in place, note the following representation-theoretic description of the product $\prod_i C_i$ of a family of coalgebras in the category thereof. 

\begin{theorem}\label{th:psdpull}
  Let $(C_i)_{i\in I}$ be a family of coalgebras and $C:=\prod C_i$ the product in the category $\tensor*[_{\Bbbk}]{\cat{Coalg}}{}$.

  \begin{enumerate}[(1),wide]
  \item The corestriction functors $\cM^{C}\to \cM^{C_i}$ induced by the product structure morphisms $C\xrightarrow{\pi_i}C_i$ realize $\cM^C$ as the pseudo-pullback of the forgetful functors $\cM^{C_i}\to \cat{Vec}$.

  \item The same goes for categories $\cM^{\bullet}_f$ of finite-dimensional comodules.

  \item The statements hold also for categories of bialgebras or Hopf algebras.  \qedhere
  \end{enumerate}
\end{theorem}
\begin{proof}
  This follows straightforwardly from standard {\it Tannaka reconstruction} \cite[Theorem 2.1.12 and Lemma 2.2.1]{schau_tann} of coalgebras and morphisms from their respective categories of comodules. In one formulation \cite[Proposition 3.3]{zbMATH01724903} particularly convenient here,
  \begin{equation*}
    \cat{Coalg}
    \ni
    C
    \xmapsto{\quad\quad}
    \left(
      \cM^C_f
      \xrightarrow{\quad\text{forgetful functor }\cat{fgt}\quad}
      \cat{Vect}_f
    \right)
  \end{equation*}
  is a right {\it biadjoint} \cite[Definition 9.8]{fiore_2cat} from
  \begin{itemize}[wide]
  \item the category of coalgebras, enhanced to a bicategory by equipping it with 2-morphisms    
    \begin{equation*}
      \begin{tikzpicture}[>=stealth,auto,baseline=(current  bounding  box.center)]
        \path[anchor=base] 
        (0,0) node (l) {$C$}
        +(4,0) node (r) {$D$}
        +(2,0) node () {\rotatebox[origin=c]{-90}{$\Rightarrow$}}
        +(2.3,0) node () {$\psi$}
        ;
        \draw[->] (l) to[bend left=20] node[pos=.5,auto] {$\scriptstyle f$} (r);
        \draw[->] (l) to[bend right=20] node[pos=.5,auto,swap] {$\scriptstyle g$} (r);
      \end{tikzpicture}
    \end{equation*}
    consisting of functionals $\psi\in C^*=\cat{Vect}(C,\Bbbk)$ intertwining the algebra morphisms $f^*$ and $g^*$, in the sense that
    \begin{equation*}
      \psi\cdot f^*(\theta)
      =
      g^*(\theta)\cdot \psi
      \quad
      \text{in the {\it dual algebra} \cite[\S 2.3]{rad}}
      \quad
      C^*
      ,\quad
      \forall \theta\in D^*. 
    \end{equation*}
  \item and a certain bicategory consisting of categories equipped with functors $\to\cat{Vect}_f$ and {\it actions} \cite[post Example 2.3]{zbMATH01724903} by the monoidal category $\cat{Vect}_f$. 
  \end{itemize}
  That right biadjoint will in particular preserve products, and checking that these are precisely the usual coalgebra products in the former bicategory and pseudo-pullbacks in the latter is straightforward. 
  
  The statement for $\cM^C$ (arbitrary, rather than finite-dimensional comodules) follows by realizing $\cM^C$ as a {\it cocompletion} \cite[Definition 1.44]{ar} of $\cM^C_f$ (every comodule being canonically the {\it filtered} \cite[Definition 1.4]{ar} union of its finite-dimensional subcomodules \cite[Finiteness Theorem 5.1.1]{mont}). 
  
  As for the bialgebra- and Hopf-algebra-flavored versions, simply observe that products in either category coincide with those of the underlying algebras. For the functors
  \begin{equation*}
    \tensor*[_{\Bbbk}]{\cat{HAlg}}{}
    \lhook\joinrel\xrightarrow{\quad\cat{incl}\quad}
    \tensor*[_{\Bbbk}]{\cat{BiAlg}}{}
    \xrightarrow{\quad\cat{fgt}\quad}
    \tensor*[_{\Bbbk}]{\cat{Coalg}}{}
  \end{equation*}
  are right adjoints \cite[Proposition 47 and Theorem 54]{zbMATH06696042} (and indeed, by \cite[Proposition 47 3.(b)]{zbMATH06696042} and \cite[Theorem 10 2.]{zbMATH06696043}, {\it monadic} in the sense of \cite[Definition 4.4.1]{brcx_hndbk-2}), so by \cite[Proposition 3.2.2]{brcx_hndbk-1} also {\it continuous} (i.e. \cite[\S V.4]{mcl_2e} preserve small limits). In particular, said functors preserve products.
\end{proof}

The following bit of language (and notation) will occasionally be useful.

\begin{definition}\label{def:famcomod}
  Let $(C_i)_i$ be a family of coalgebras. The corresponding category $\cM^{(C_i)}$ of {\it $(C_i)$-comodules} is by definition that of $C$-comodules for the product $C:=\prod C_i$ in the category $\cat{Coalg}$. Per \Cref{th:psdpull}, it consists of vector spaces
  \begin{itemize}[wide]
  \item carrying separate $C_i$-comodule structures for all $i$;

  \item with those structures leaving invariant an exhaustive filtration by finite-dimensional subspaces.
  \end{itemize}
  Given the product coincidence for coalgebras, bialgebras and Hopf algebras (remarked upon in the course of the proof of \Cref{th:psdpull}), the terminology just introduced will be used unambiguously for bialgebras and Hopf algebras as well (as in speaking of $(H_i)$-comodules for $H_i\in \cat{HAlg}$, etc.).
\end{definition}

\begin{theorem}\label{th:infdimprod}
  For a family $(H_i)_{i\in I}$ of bialgebras over an infinite field the following conditions are equivalent.
  \begin{enumerate}[(a),wide]
  \item\label{item:th:infdimprod:almosttriv} The $H_i$ are finite-dimensional, with at most one of dimension $>1$.
    
  \item\label{item:th:infdimprod:findim} The product $H:=\prod_i H_i$ in the category of bialgebras is finite-dimensional. 
  \item\label{item:th:infdimprod:finsimpl} $H$ has finitely many (isomorphism classes of) simple comodules.
  \item\label{item:th:infdimprod:strongfinsimpl} For every $d\in \bZ_{\ge 0}$, $H$ has finitely many (isomorphism classes of) simple $d$-dimensional comodules.
  \end{enumerate}
\end{theorem}

For both \Cref{th:infdimprod} and later in \Cref{th:coalgfdprod}, when discussing coalgebras, it will be convenient to have to handle binary families of objects only. 

\begin{lemma}\label{le:binfamenough}
  Let $(C_i)_{i\in I}$ be a family of $\Bbbk$-coalgebras. 
  \begin{enumerate}[(1),wide]
  \item\label{item:le:binfamenough:coalg} Suppose $I$ is finite and the coalgebra product $\prod_I C_i$
    \begin{itemize}[wide]
    \item either is finite dimensional;

    \item or has finitely many simple comodules;

    \item or has finitely many simple comodules in any given dimension. 
    \end{itemize}
    Then the same, respectively, goes for every product $\prod_J C_j$ over every subfamily $J\subseteq I$. 

  \item\label{item:le:binfamenough:bialg} If the $C_i$ are bialgebras (or Hopf algebras) then the above holds for arbitrary, possibly infinite families. 
  \end{enumerate}
\end{lemma}
\begin{proof}
  Both remarks are simple enough, but the second is particularly straightforward. As in all categories $\cC$ with {\it zero objects} (both initial and final \cite[pre Proposition 4.5.16]{brcx_hndbk-1}), in both $\cat{BiAlg}$ and $\cat{HAlg}$ the canonical morphisms $\prod_I c_i \to \prod_J c_j$, $J\subseteq I$ to products of subfamilies split:
  \begin{equation*}
    \begin{tikzpicture}[>=stealth,auto,baseline=(current  bounding  box.center)]
      \path[anchor=base] 
      (0,0) node (l) {$\prod_{j\in J}c_j$}
      +(4,.5) node (u) {$\prod_{i\in I}c_i$}
      +(8,0) node (r) {$\prod_{j\in J}c_j$}
      ;
      \draw[->] (l) to[bend left=6] node[pos=.5,auto] {$\scriptstyle (\iota_i)_{i\in I}$} (u);
      \draw[->] (u) to[bend left=6] node[pos=.5,auto] {$\scriptstyle (\text{canonical projections }\pi_j)_{j\in J}$} (r);
      \draw[->] (l) to[bend right=6] node[pos=.5,auto,swap] {$\scriptstyle \id$} (r);
    \end{tikzpicture}
  \end{equation*}
  for any subset $J\subseteq I$, with
  \begin{equation*}
    \iota_i
    :=
    \begin{cases}
      \text{product structure map }\pi_i
      &\text{if }i\in J\\
      \text{zero morphism}
      &\text{otherwise}.
    \end{cases}
  \end{equation*}
  $\prod_J$ is thus a retract of $\prod_I$, and the properties in the statement are plainly inherited by retracts. We thus turn to \Cref{item:le:binfamenough:coalg}. 

  Fix a non-zero finite-dimensional $(C_l)_{l\in I\setminus J}$-comodule $W$ (in the sense of \Cref{def:famcomod}); one such exists, given the finiteness of $I$ (the finiteness of $I\setminus J$ suffices). If $(V_t)_t$ are infinitely many simple $(C_j)_{j\in J}$-comodules (of bounded dimension) then $(V_t\otimes W)_t$ carry obvious $(C_i)_{i\in I}$-comodule (respectively of bounded dimension). They are semisimple and {\it isotypic} (i.e. sums of copies of single simple comodules) over $(C_j)_{j\in J}$, so some infinite family (respectively of bounded dimension) of mutually distinct simple $(C_i)_I$-comodules can be extracted from among their subquotients.

  This, so far, takes care of the two last bullet points in \Cref{item:le:binfamenough:coalg}. The first can be handled similarly, noting that infinite-dimensionality will provide an infinite family $(V_t)$ of mutually non-isomorphic $(C_j)_J$-comodules that are either
  \begin{itemize}
  \item simple;
  \item or non-split extensions of two fixed simple $(C_j)_J$-comodules. 
  \end{itemize}
  Either way, the argument above applies. 
\end{proof}

As \Cref{res:notcoalgs}\Cref{item:res:notcoalgs:notcoalg} below observes, one cannot, generally, drop the finiteness requirement in \Cref{le:binfamenough}\Cref{item:le:binfamenough:coalg}. 

\pf{th:infdimprod}
\begin{th:infdimprod}  
  The implications
  \begin{equation}\label{eq:easyimpl}
    \text{\Cref{item:th:infdimprod:almosttriv}}
    \xRightarrow{\quad\text{1-dimensional object is terminal}\quad}
    \text{\Cref{item:th:infdimprod:findim}}
    \xRightarrow{\quad}
    \text{\Cref{item:th:infdimprod:finsimpl}}
    \xRightarrow{\quad}
    \text{\Cref{item:th:infdimprod:strongfinsimpl}}
  \end{equation}
  being obvious, we focus, for the duration of the proof, on \Cref{item:th:infdimprod:strongfinsimpl} $\Rightarrow$ \Cref{item:th:infdimprod:almosttriv}.
  
  It is enough, by \Cref{le:binfamenough}\Cref{item:le:binfamenough:bialg}, to consider binary products $H:=H_{\alpha}\times H_{\beta}$ of finite-dimensional bialgebras $H_i$, $i=\alpha,\beta$. Throughout, we use the description of (finite-dimensional) $(H_i)_i$-comodules implicit in \Cref{th:psdpull}: vector spaces equipped with separate $H_i$-comodule structures. More specifically, this means $H_i$-comodules $V_i$ together with linear isomorphisms $H_i\xrightarrow[\cong]{\varphi_i}V$ to a common vector space $V$, morphisms $(V,V_i,\varphi)\to (V',V_i',\varphi'_i)$ being commutative diagrams
  \begin{equation*}
    \begin{tikzpicture}[>=stealth,auto,baseline=(current  bounding  box.center)]
      \path[anchor=base] 
      (0,0) node (l) {$V_i$}
      +(2,.5) node (u) {$V$}
      +(2,-.5) node (d) {$V'_i$}
      +(4,0) node (r) {$V'$,}
      ;
      \draw[->] (l) to[bend left=6] node[pos=.5,auto] {$\scriptstyle \varphi_i$} node[pos=.5,auto,swap] {$\scriptstyle \cong$} (u);
      \draw[->] (u) to[bend left=6] node[pos=.5,auto] {$\scriptstyle \psi\in\cat{Vec}$} (r);
      \draw[->] (l) to[bend right=6] node[pos=.5,auto,swap] {$\scriptstyle \psi_i\in \cM^{H_i}$} (d);
      \draw[->] (d) to[bend right=6] node[pos=.5,auto,swap] {$\scriptstyle \varphi'_i$} node[pos=.5,auto] {$\scriptstyle \cong$} (r);
    \end{tikzpicture}
  \end{equation*}
  (same $\psi$ for both $i=\alpha,\beta$). There are several cases to consider, along with a more broadly-applicable preamble.

  \begin{enumerate}[(I),wide]
  \item{\bf A general argument.} We will construct $(H_i)_{i=\alpha,\beta}$-comodules $W$ equipped with isomorphisms $W_i\xrightarrow[\cong]{\varphi_i}W$, $W_i\in \cM_f^{H_i}$, $i=\alpha,\beta$ varying the choices involved. Specifically, in each case considered below the $W_i$ will be fixed and the $\varphi_i$ will vary. Those choices will always ensure that $W$ is simple over $H:=H_\alpha\times H_{\beta}$ (bialgebra product), so it will suffice to argue that there are infinitely many mutually non-isomorphic choices of $\varphi_i$.

    Upon fixing a pair $(\varphi_{i,0})_{}i=\alpha,\beta$, , arbitrary pairs $(\varphi_i)_i$ are in bijection with the group $GL(W)^2$: simply compose the selected $\varphi_{i,0}$ separately with elements of the general linear group of $GL(W)$. We will regard $GL(W)$ as the group of {\it $\Bbbk$-points} of a {\it linear algebraic ($\Bbbk$-)group} \cite[\S I.1.1]{brl}: the $A$-points, for commutative $\Bbbk$-algebras $A$, are defined as $GL(W_{A})$. Bold fonts will distinguish between schemes (as in ${\bf GL}(W)$) and the corresponding groups of $\Bbbk$-points (such as $GL(W)$). We similarly have groups $GL_i(W_i)$, $i=\alpha,\beta$ of comodule automorphisms, as $\Bbbk$-points of algebraic groups ${\bf GL}_i(W_i)$.
    
    We will furthermore assume $\Bbbk$ is large enough for Jordan-H\"older filtrations of the $W_{\alpha}$ involved to be stable (i.e. the simple subquotients stay simple after extending scalars to $\overline{\Bbbk}$). This is always achievable by extending scalars along some {\it finite} field extension $\Bbbk\le \Bbbk'$, and such extensions do not affect products (\Cref{th:finextradj} below, where we discuss the compatibility between field extensions and universal constructions). 
    
    The fixed $\varphi_{i,0}$ transport comodule automorphisms over to $W$, realizing embeddings ${\bf GL}_i(W_i)\to {\bf GL}(W)$. The set of isomorphism classes of $(\varphi_{\alpha,\beta})\in \cM^{(H_{\alpha,\beta})}$ is then identifiable with the double coset space
    \begin{equation}\label{eq:2coset}
      GL(W)\backslash GL(W)^2/GL_{\alpha}(W_{\alpha})\times GL_{\beta}(W_{\beta}),
    \end{equation}
    with the left-hand $GL$ acting on $(\varphi_{\alpha,\beta})$ by simultaneous composition (i.e. via the diagonal embedding $GL\to GL^2$) and the right-hand $GL_{\alpha}\times GL_{\beta}$ acting by precomposition in the two $GL(W)$ factors. To compress the notation, set
    \begin{equation*}
      \begin{aligned}
        {\bf X}&:={\bf GL}(W)^2
                 \quad(\text{the scheme in \Cref{eq:2coset} being acted upon})\\
        {\bf G}&:={\bf GL}(W)\times {\bf GL_{\alpha}}(W_{\alpha}) \times {\bf GL_{\beta}}(W_{\beta})
                 \quad(\text{the algebraic group in \Cref{eq:2coset} doing the acting})\\
      \end{aligned}
    \end{equation*}
    Orbit spaces such as the
    \begin{equation*}
      X/G
      ,\quad
      X:={\bf X}(\Bbbk)
      ,\quad
      G:={\bf G}(\Bbbk)
    \end{equation*}
    of \Cref{eq:2coset} might not be the space of $\Bbbk$-points of a scheme (the familiar issues attending the formation of {\it geometric quotient} \cite[\S\S 6.1 and 6.16]{brl}, \cite[Definition 0.6]{fkm}, etc.). This can nevertheless be remedied \cite[Theorem and opening paragraph]{zbMATH03199301} upon passing to a dense open ${\bf G}$-invariant ${\bf U}\subseteq {\bf X}$. Because furthermore the resulting quotient ${\bf U}/{\bf G}$ will be {\it unirational} \cite[\S 13.7]{brl} (being surjected upon by a dense open subspace ${\bf U}$ of affine space), the infinitude of ${\Bbbk}$ implies \cite[\S 13.7, concluding sentence]{brl} that ${\bf U}(\Bbbk)/{\bf G}(\Bbbk)$ is infinite (and hence so is \Cref{eq:2coset}) provided it has positive dimension. Addressing that positivity hinges on (what we refer to as) the {\it crude dimension (or size) estimate}
    \begin{equation}\label{eq:crudeest}
      \dim {\bf X}>\dim{\bf G}-1
      \quad\text{or}\quad
      \dim GL(W)
      >
      \dim GL_{\alpha}(W_{\alpha})+\dim GL_{\beta}(W_{\beta})-1, 
    \end{equation}
    the $-1$ to account for the fact that diagonal scaling by the left-hand $GL$ in \Cref{eq:2coset} coincides with scaling by the diagonal copy of $\Bbbk$ in $GL_{\alpha}\times GL_{\beta}$ on the right. We next tackle the aforementioned cases.

  \item {\bf Each $H_i$, $i=\alpha,\beta$ has a simple comodule $V_i$ of respective dimension $d_i\ge 2$.} The $W_i$ of the above sketch will be $V_i^{\oplus d/d_i}$ respectively, with $d=\lcm(d_{\alpha},d_{\beta})$. The common vector space receiving the isomorphisms $W_i\xrightarrow[\cong]{\varphi_i}W$ is thus $d$-dimensional.

    The resulting $(H_i)$-comodule structures are indeed simple, for any non-zero subcomodule must restrict to a sum of copies of each $V_i$, and hence have dimension divisible by both $d_i$ so also by $d$. The crude estimate \Cref{eq:crudeest} reads
    \begin{equation*}
      d^2
      >
      \left(\frac d{d_{\alpha}}\right)^2
      +
      \left(\frac d{d_{\beta}}\right)^2
      -1
      \quad\text{i.e.}\quad
      1
      >
      \frac 1{d_{\alpha}^2}
      +
      \frac 1{d_{\beta}^2}
      -\frac 1{d^2},
    \end{equation*}
    indeed valid by the assumption that $d_{\alpha,\beta}\ge 2$ (note that the $-1$ correction of \Cref{eq:crudeest} is not even needed in this case).

  \item {\bf $H_{\alpha}$ has a simple comodule $V_{\alpha}$ of dimension $d_{\alpha}\ge 2$ and $H_{\beta}$ does not.} $H_{\beta}$, in other words, is {\it pointed} \cite[Definition 3.4.4]{rad}: its simple right comodules are all 1-dimensional. Set $d:=\lcm(d_{\alpha},2)$ (so $d=d_{\alpha}$ if the latter is even and $2d_{\alpha}$ if not). $W_{\alpha}\cong V_{\alpha}^{\oplus d/d_{\alpha}}$ again, while $W_{\beta}$ will be a sum of $\frac d2$ copies of a single 2-dimensional $H_{\beta}$-comodule
    \begin{equation}\label{eq:2dimv}
      V_{\beta}
      :=
      \begin{cases}
        V_{\beta,0}\oplus V_{\beta,1}
        &\quad\dim V_{\beta,s}=1,\ V_{\beta,s}\text{ non-isomorphic if they exist}\\
        &\phantom{-}\\
        \text{\parbox{4cm}{a non-split self-extension of a 1-dimensional comodule $V_{\beta,0}$}}
        &\quad\text{otherwise}
      \end{cases}
    \end{equation}
    If $d_{\alpha}$ is even then $W$ is certainly simple as an $H_{\alpha,\beta}$-comodule, for over $H_{\alpha}$ it is $W_{\alpha}$. Otherwise (assuming $d_{\alpha}$ odd, that is), we can force simplicity as follows.

    Having fixed a specific decomposition $W\cong W_{\alpha}\cong V_{\alpha}^2$, all {\it other} copies of $V_{\alpha}$ embedded therein as $H_{\alpha}$-subcomodules are precisely the graphs $\Gamma_{\lambda}$, $\lambda\in \Bbbk^{\times}$ of the scaling maps from one distinguished $V_{\alpha}$ summand to the other. Any non-zero $H_{\alpha,\beta}$-subcomodule of $W$ will of course intersect its {\it socle} (i.e. \cite[\S 3.1]{dnr} largest semisimple subobject) $W_{\beta}=V_{\beta,0}^{\frac d2}$ as an $H_{\beta}$-comodule. Restricting to an open dense choice of $\varphi_{\alpha,\beta}$, we can assume said socle is similarly the graph $\Gamma_{T}$ of an isomorphism
    \begin{equation*}
      V_{\alpha}
      \xrightarrow[\quad\cong\quad]{\quad T=T_{\varphi_{\alpha,\beta}}\quad}
      V_{\alpha}
    \end{equation*}
    The elements of that socle belonging to some $H_{\alpha}$-subcomodule $V_{\alpha}$, then, are (identifiable with) the eigenspaces of $T$. Now, because $d_{\alpha}\ge 3$ (being $\ge 2$ and odd), a non-zero $H_{\alpha,\beta}$-subcomodule, if proper, must
    \begin{itemize}[wide]
    \item be isomorphic to $V_{\alpha}$ as an $H_{\alpha}$-comodule, coinciding with one of the $\Gamma_{\lambda}$, $\lambda\in \Bbbk^{\times}$;
    \item and also intersect $\Gamma_T$ along at least a {\it two}-dimensional space. 
    \end{itemize}
    To avoid this, it is enough to range over those choices of $\varphi_{\alpha,\beta}$ for which the corresponding $T$ has no repeated eigenvalues (an open dense condition).
    
    Having thus ensured that $W$ is simple, it remains to verify the estimate \Cref{eq:crudeest}; here, it translates to
    \begin{equation}\label{eq:dalphaeo}
      \begin{cases}
        d_{\alpha}^2>1+2\left(\frac {d_{\alpha}}2\right)^2-1=\frac{d_{\alpha}^2}2
        \iff
        d_{\alpha}>0
        &\text{if $d_{\alpha}$ is even}\\
        (2d_{\alpha})^2 > 4+2d_{\alpha}^2-1 = 3+2d_{\alpha}^2
        \iff
        d_{\alpha}>1
        &\text{if $d_{\alpha}$ is odd}
      \end{cases}
    \end{equation}
    Both inequalities, of course, hold. And once more, the $-1$ terms on the right-hand side is not needed, on either line of \Cref{eq:dalphaeo}. 
    
  \item {\bf $H_{\alpha,\beta}$ are both pointed.} $W$ will now be 2-dimensional, and its comodule structures $W_{\alpha}$ and $W_{\beta}$ are as in \Cref{eq:2dimv} (for $\alpha$, now, as well as $\beta$).

    $W$ will be simple as soon as the socles $V_{i,0}\subset V_i$, $i=\alpha,\beta$ are distinct, an open dense condition in the space of choices of $(\varphi_{\alpha,\beta})$. As for the size comparison \Cref{eq:crudeest}, here it is $4>2+2-1$; indeed valid (with the $-1$ term needed, for once).
  \end{enumerate}  
\end{th:infdimprod}

\begin{remarks}\label{res:notcoalgs}
  \begin{enumerate}[(1),wide]

  \item\label{item:res:notcoalgs:notunbdddim} The infinitely many simple $(H_i)$-comodules provided by \Cref{th:infdimprod} for binary (say) families $\dim H_{\alpha,\beta}>1$ certainly need not have unbounded dimensions, in general: see \Cref{ex:infdih} below. 
    
  \item\label{item:res:notcoalgs:notcoalg} Despite $\cat{Coalg}$ having ``the same'' products as $\cat{BiAlg}$ and $\cat{HAlg}$ (proof of \Cref{th:psdpull}), the statement of \Cref{th:infdimprod} will not go through in that generality for coalgebras: if a (necessarily infinite) coalgebra family $C_i$ is such that
    \begin{equation*}
      \sup_i \inf_{0\ne V\in \cM^{C_i}}\dim V = \infty
    \end{equation*}
    (e.g. if $C_i$ is a family of matrix coalgebras $M_{n_i}^*$ with $\{n_i\}_i$ unbounded) then no finite-dimensional vector space can support $C_i$-comodule structures simultaneously for all $i$. By the {\it comodule finiteness theorem} \cite[5.1.1]{mont} the coalgebra product $\prod C_i$ must be trivial (i.e. the 0 coalgebra).

  \item\label{item:res:notcoalgs:alg} Note the contrast between the situation of \Cref{item:res:notcoalgs:notcoalg} above and the dual version: {\it algebras} over a field always embed into their coproduct in $\tensor[_{\Bbbk}]{\cat{Alg}}{}$, for instance by the standard {\it normal-form} description of that coproduct \cite[Corollary 8.1]{zbMATH03510490}.

  \item To pursue the duality issue, the counterpart to \Cref{item:res:notcoalgs:notcoalg} is the remark that the coproducts of bialgebras or Hopf algebras are the same as those of the underlying algebras. The dual version of \Cref{th:infdimprod}, moreover, is easily checked (again standard, via normal forms). 

    Having observed (in the proof of \Cref{le:binfamenough}) that $\cat{BiAlg}$ and $\cat{HAlg}$ have zero objects, there are canonical morphisms
    \begin{equation}\label{eq:coprod2prod}
      \coprod_i H_i
      \xrightarrow{\quad}
      \prod_i H_i.
    \end{equation}
    These do not appear to be, however, especially useful in deducing \Cref{th:infdimprod} from its dual variant: \Cref{eq:coprod2prod} are not, generally, embeddings (so that even if the domain is infinite-dimensional, it is not clear, a priori, that this entails the infinite dimensionality of the codomain).

    To see this, consider what happens when all bialgebras $H_i$ happen to be commutative: the subcategories
    \begin{equation*}
      \tensor*[_{\Bbbk}]{\cat{BiAlg}}{_c}
      \subset
      \tensor*[_{\Bbbk}]{\cat{BiAlg}}{}
      \quad\text{and}\quad
      \tensor*[_{\Bbbk}]{\cat{HAlg}}{_c}
      \subset
      \tensor*[_{\Bbbk}]{\cat{HAlg}}{}
    \end{equation*}
    of commutative objects being {\it reflective} by \cite[diagram (9)]{zbMATH06696043} (i.e. \cite[Definition 3.5.2]{brcx_hndbk-1} full, with the inclusions being right adjoints), they are automatically closed under limits \cite[Proposition 3.5.3]{brcx_hndbk-1}. In particular, the bialgebra product of commutative bialgebras is automatically commutative. Naturally, their coproduct will not be (save for trivial particular cases).
  \end{enumerate}
\end{remarks}

\begin{example}\label{ex:infdih}
  Take for both $H_{\alpha,\beta}$ the function algebra on the group $\bZ/2$:
  \begin{equation*}
    H_{\alpha}:=\Bbbk^{\bZ/2}=:H_{\beta},
  \end{equation*}
  with the pointwise multiplication and the coalgebra structure dualizing the multiplication on $\bZ/2$. A finite-dimensional $(H_i)$-comodule is then nothing but a finite-dimensional representation (over $\Bbbk$) of the {\it infinite dihedral group} \cite[Exercise 11.63]{rot-gp}
  \begin{equation*}
    \text{group coproduct }(\bZ/2)*(\bZ/2)
    \cong
    \bZ\rtimes \bZ/2
  \end{equation*}
  ({\it semidirect product} \cite[Theorem 7.2 and Definition preceding it]{rot-gp}) for the inversion $\bZ/2$-action on $\bZ$.

  For algebraically closed $\Bbbk$ irreducible finite-dimensional $(\bZ\rtimes \bZ/2)$-modules are at most 2-dimensional: a generator for $\bZ$ will have a $\lambda$-eigenvector for some $\lambda\in \Bbbk^{\times}$, and the $\bZ/2$-factor will swap such an eigenvector with one with eigenvalue $\lambda^{-1}$. 
\end{example}

\Cref{th:infdimprod} cannot quite go through in precisely the same form over finite fields, but the needed adjustments are not overly drastic.

\begin{corollary}\label{cor:finfld}
  If the ground field $\Bbbk$ is finite, the conditions of \Cref{th:infdimprod} satisfy
  \begin{equation*}
    \text{\Cref{item:th:infdimprod:almosttriv}}
    \xLeftrightarrow{\quad}
    \text{\Cref{item:th:infdimprod:findim}}
    \xLeftrightarrow{\quad}
    \text{\Cref{item:th:infdimprod:finsimpl}}
    \xRightarrow{\quad}
    \text{\Cref{item:th:infdimprod:strongfinsimpl}}.
  \end{equation*}
  The last condition is strictly weaker, holding whenever all $H_i$ are finite-dimensional, with at most finitely many of dimension $>1$. 
\end{corollary}
\begin{proof}
  We of course still have \Cref{eq:easyimpl} in any case, and the last claim is simple: a $d$-dimensional vector space over a finite field is of course finite, and supports finitely many joint comodule structures over the finitely many $(>1)$-dimensional, finite-dimensional $H_i$.

  It thus remains to argue that the implication \Cref{item:th:infdimprod:finsimpl} $\Rightarrow$ \Cref{item:th:infdimprod:almosttriv} (equivalently, its contrapositive) still holds over finite fields, restricting the discussion, once again, to pairs of bialgebras. 

  If $H_{\alpha,\beta}$ both have dimension $>1$, \Cref{th:infdimprod} shows that for some $d\in \bZ_{\ge 0}$ a $d$-dimensional $\overline{\Bbbk}$-vector space $V$ over the algebraic closure $\overline{\Bbbk}\supset \Bbbk$ supports infinitely many mutually non-isomorphic pairs of $(H_{\alpha,\beta})_{\overline{\Bbbk}}$-comodule structures (i.e. $H$-comodule structures $V_t$). Observe that in general, for any $\Bbbk$-coalgebra $C$, every finite-dimensional (or simple) $C_{\overline{\Bbbk}}$-comodule arises as a subquotient of a comodule of the form
  \begin{equation*}
    W_{\overline{\Bbbk}}
    ,\quad
    W\in \tensor*[_{\Bbbk}]{\cM}{^{C}}
    \quad\text{finite-dimensional (respectively simple)}.
  \end{equation*}
  Applying this to $C=H_{\alpha,\beta}$, there must be $H$-comodules (necessarily mutually non-isomorphic, and possibly of differing, possibly unbounded dimensions) recovering the $V_t$. 
\end{proof}

As to coalgebras (as opposed to bialgebras), the phenomenon noted in \Cref{res:notcoalgs}\Cref{item:res:notcoalgs:notcoalg}, necessarily requiring {\it infinite} families, turns out to be the only impediment to \Cref{th:infdimprod} (and \Cref{cor:finfld}) going through in full generality.

\begin{theorem}\label{th:coalgfdprod}
  The conclusions of \Cref{th:infdimprod} and \Cref{cor:finfld} hold in $\tensor*[_{\Bbbk}]{\cat{Coalg}}{}$ for finite families of coalgebras. 
\end{theorem}
\begin{proof}
  Having reduced the matter to binary families by \Cref{le:binfamenough}\Cref{item:le:binfamenough:coalg}, the proofs of \Cref{th:infdimprod} and \Cref{cor:finfld} apply verbatim. 
\end{proof}

\section{Transporting universal constructions along field extensions}\label{se:fldext}

The proof of \Cref{th:infdimprod} remarks in passing on the behavior of coalgebra (bialgebra, Hopf algebra) products under extending the ground field; it is perhaps apposite then, at this stage, to point out an error in the literature. \cite[Theorem 4.1(2)]{MR4712418} claims that scalar extension $(-)_{\Bbbk'}=-\otimes_{\Bbbk}\Bbbk'$ along a field extension $\Bbbk\le \Bbbk'$ is a right adjoint between the respective categories $\cat{Coalg}$, $\cat{BiAlg}$ and $\cat{HAlg}$. This is not so: the functor preserves equalizers (\cite[Lemma 4.2]{MR4712418} appears to be correct), but not, in general, arbitrary products (there is a flaw in the proof of \cite[Lemma 4.3]{MR4712418}).

That in general infinite coalgebra products cannot be preserved by infinite field extensions can be seen as follows.

\begin{example}\label{ex:scalewrong}
  Recall the {\it cofree coalgebra} construction
  \begin{equation*}
    \tensor*[_{\Bbbk}]{\cat{Vect}}{}
    \ni
    V
    \xmapsto{\quad}
    C_{\cat{cf}}(V)
    \in
    \tensor*[_{\Bbbk}]{\cat{Coalg}}{}
  \end{equation*}
  of \cite[\S 1.6]{dnr} (or \cite[\S 6.4]{swe}): the right adjoint to the forgetful functor $\cat{Coalg}\to \cat{Vect}$. Being a right adjoint it preserves limits, and hence also products. In the commutative diagram
  \begin{equation*}
    \begin{tikzpicture}[>=stealth,auto,baseline=(current  bounding  box.center)]
      \path[anchor=base] 
      (0,0) node (l) {$\left(\prod_i C_{\cat{cf}}(V_i)\right)_{\Bbbk'}$}
      +(4,.5) node (u) {$\prod_i C_{\cat{cf}}(V_i)_{\Bbbk'}$}
      +(4,-.5) node (d) {$\left(\prod_i V_i\right)_{\Bbbk'}$}
      +(8,0) node (r) {$\prod_i V_{i,\Bbbk'}$}
      ;
      \draw[->] (l) to[bend left=6] node[pos=.5,auto] {$\scriptstyle $} (u);
      \draw[->>] (u) to[bend left=6] node[pos=.5,auto] {$\scriptstyle $} (r);
      \draw[->>] (l) to[bend right=6] node[pos=.5,auto,swap] {$\scriptstyle $} (d);
      \draw[->] (d) to[bend right=6] node[pos=.5,auto,swap] {$\scriptstyle $} (r);
    \end{tikzpicture}
  \end{equation*}
  the double-headed arrows are indeed surjections \cite[Exercise 1.6.2]{dnr}, whereas the right-hand bottom arrow {\it never} is for infinite $\Bbbk\le \Bbbk'$ and infinitely many non-vanishing $V_i$. The top left-hand arrow thus fails to be an isomorphism (indeed, surjective) in all of those cases (infinite field extensions, infinite vector-space families).

  Even more drastically, it is possible, for infinite field extensions $\Bbbk\le \Bbbk'$, for the coalgebra product of $C_i$ to vanish while that of $(C_{i})_{\Bbbk'}$ does not (\Cref{ex:dualfields}). 
\end{example}

\begin{example}\label{ex:dualfields}
  Suppose the algebraic closure $\overline{\Bbbk}\supset \Bbbk$ is infinite (equivalently, by the celebrated {\it Artin-Schreier theorem} \cite[Theorem 11.14]{jcbsn_alg-2}, $\Bbbk$ is neither algebraically closed nor {\it real closed} in the sense of \cite[\S 5.1]{jcbsn_alg-1}, \cite[p.87]{lam_1st_2e_2001}, etc.).

  Let $\Bbbk_i\supset \Bbbk$ be finite extensions with $\{[\Bbbk_i:\Bbbk]\}_i$ unbounded, and set $C_i:=\Bbbk_i^*$, the dual $\Bbbk$-coalgebra. No finite-dimensional $\Bbbk$-vector space carries $\Bbbk_i$-linear structures for all $i$ simultaneously, so the coalgebra product $\prod_i C_i$ vanishes as in \Cref{res:notcoalgs}\Cref{item:res:notcoalgs:notcoalg}. On the other hand, $\Bbbk_i\otimes_{\Bbbk}\overline{\Bbbk}$ all have 1-dimensional modules (by the {\it Hilbert Nullstellensatz} \cite[Corollary 7.10]{am_comm}, for they are all finite-dimensional commutative algebras over an algebraically closed field). The coalgebra product $\prod_i (C_i)_{\overline{\Bbbk}}$ thus has at least a 1-dimensional comodule, so cannot vanish. 
\end{example}

\begin{remark}\label{re:evenfinnotpres}
  It will become clear later (\Cref{th:algextradj}) that even {\it finite} coalgebra products are not preserved by arbitrary field extensions. 
\end{remark}

The following simple strengthening of \cite[Exercise 1.6.2]{dnr} recalled in \Cref{ex:scalewrong} above (the surjectivity of the canonical map $C_{\cat{cf}}(V)\to V$ from the cofree coalgebra) will be useful later. The {\it cofree Hopf algebra} $H_{\cat{cf}}(A)\to A$ on an algebra, referred to in the statement, is (the image of $A$ through) the right adjoint to the forgetful functor $\cat{HAlg}\to \cat{Alg}$. That forgetful functor is a left adjoint by, say, \cite[Theorem 2.5]{agore_hopf}, and indeed (by \cite[Theorem 10 1.]{zbMATH06696043}) {\it comonadic} (the notion categorically dual to \cite[Definition 4.4.1]{brcx_hndbk-2}).

\begin{lemma}\label{le:hcfsurj}
  For any $\Bbbk$-algebra $A$ the cofree Hopf algebra structure map $H_{\cat{cf}}(A)\to A$ is onto. 
\end{lemma}
\begin{proof}
  Were it not, all algebra maps into $A$ with Hopf domain would factor through some proper subspace of $A$ (one subspace, valid for all Hopf algebras). On the other hand though, every $a\in A$ is the image of the algebra morphism
  \begin{equation*}
    \Bbbk[x]
    \ni x
    \xmapsto{\quad}
    a
    \in A.
  \end{equation*}
  The polynomial ring can be made into a Hopf algebra by declaring $x$ {\it primitive} (i.e. \cite[Definition 1.3.4(b)]{mont} $\Delta(x)=x\otimes 1+1\otimes x$), and we are done. 
\end{proof}

\begin{remark}\label{re:monadlift}
  The existence of the cofree Hopf algebra $H_{\cat{cf}}$ (and its bialgebra analogue $B_{\cat{cf}}(A)$ for an algebra $A$: nothing but $C_{\cat{cf}}(A)$, equipped with a natural bialgebra structure \cite[\S VI, pp.134-135]{swe}) is also an instance of {\it monadic adjunction lifting}: apply \cite[Theorem 4.5.6]{brcx_hndbk-2} to the forgetful diagram(s)
  \begin{equation*}
    \begin{tikzpicture}[>=stealth,auto,baseline=(current  bounding  box.center)]
      \path[anchor=base] 
      (0,0) node (l) {$\cat{BiAlg}$ or $\cat{HAlg}$}
      +(3,.5) node (u) {$\cat{Alg}$}
      +(3,-.5) node (d) {$\cat{Coalg}$}
      +(5,0) node (r) {$\cat{Vect}$}
      ;
      \draw[->] (l) to[bend left=6] node[pos=.5,auto] {$\scriptstyle $} (u);
      \draw[->] (u) to[bend left=6] node[pos=.5,auto] {$\scriptstyle $} (r);
      \draw[->] (l) to[bend right=6] node[pos=.5,auto,swap] {$\scriptstyle $} (d);
      \draw[->] (d) to[bend right=6] node[pos=.5,auto,swap] {$\scriptstyle $} (r);
    \end{tikzpicture}
  \end{equation*}
  with monadic \cite[Proposition 47, points 1. and 3.]{zbMATH06696042} and \cite[Theorem 10 2.]{zbMATH06696043} south-eastward functors.
\end{remark}

It is no accident that in \Cref{ex:dualfields} and \Cref{ex:scalewrong} the field extensions had to be infinite, for \cite[Theorem 4.1(2)]{MR4712418} does, at least, hold for finite extensions.

\begin{theorem}\label{th:finextradj}
  The following conditions on a field extension $\Bbbk\le \Bbbk'$ are equivalent.
  \begin{enumerate}[(a),wide]
  \item\label{item:th:finextradj:fin} The extension is finite. 

  \item\label{item:th:finextradj:coalgadj} The corresponding 
    \begin{equation}\label{eq:scalextcoalg}
      \tensor*[_{\Bbbk}]{\cat{Coalg}}{}
      \xrightarrow{\quad(-)_{\Bbbk'}=-\otimes_{\Bbbk}\Bbbk'\quad}
      \tensor*[_{\Bbbk'}]{\cat{Coalg}}{}
    \end{equation}
    is a right adjoint.

  \item\label{item:th:finextradj:coalgcont} \Cref{eq:scalextcoalg} is continuous or, equivalently, preserves products.

  \item\label{item:th:finextradj:coalgpow} \Cref{eq:scalextcoalg} preserves powers (i.e. products of copies of a single coalgebra).

  \item\label{item:th:finextradj:bhalgadj} The analogue of \Cref{item:th:finextradj:coalgadj} for categories of bialgebras and/or Hopf algebras.

  \item\label{item:th:finextradj:bhalgcont} The analogue of \Cref{item:th:finextradj:coalgcont} for categories of bialgebras and/or Hopf algebras.

  \item\label{item:th:finextradj:bhalgpow} The analogue of \Cref{item:th:finextradj:coalgpow} for categories of bialgebras and/or Hopf algebras.

  \item\label{item:th:finextradj:bhalgspecpow} $(-)_{\Bbbk'}$ preserves the $\aleph_0$-power of the cofree Hopf algebra on the ground field, regarded as an algebra. 
  \end{enumerate}
\end{theorem}
\begin{proof}    
  We tackle the various implications in turn. 
  \begin{enumerate}[label={},wide]
  \item {\bf \Cref{item:th:finextradj:coalgadj} $\Leftrightarrow$ \Cref{item:th:finextradj:coalgcont}:} Regardless of finiteness, \Cref{eq:scalextcoalg} is in any case a {\it left} adjoint or, equivalently (by {\it Freyd's special adjoint functor theorem} \cite[\S 0.7]{ar}, for $\cat{Coalg}$ is a locally presentable \cite[Definition 1.17]{ar} category by \cite[\S 2.7, Proposition, item 1.]{zbMATH05312006} or \cite[Lemma 1 2.]{zbMATH06696043}), cocontinuous. Indeed, scalar extension is left adjoint to scalar restriction at the level of vector spaces, and colimits of coalgebras are those of underlying vector spaces ($\cat{Coalg}\xrightarrow{\text{forget}}\cat{Vect}$ being left adjoint \cite[\S 2.7, Proposition, item 2.]{zbMATH05312006}).

    As a cocontinuous functor between locally presentable categories, \Cref{eq:scalextcoalg} will be a right adjoint as soon as it is also continuous by \cite[Theorem 1.66]{ar}. And continuity is in turn equivalent to product preservation (by \cite[Proposition 13.4]{ahs}), for equalizers are preserved in any case, by possibly infinite field extensions, by \cite[Lemma 4.2 and discussion pertaining to equalizers in the proof of Theorem 4.1]{MR4712418}. 

  \item {\bf \Cref{item:th:finextradj:bhalgadj} $\Leftrightarrow$ \Cref{item:th:finextradj:bhalgcont}:} The argument just given works here too, substituting
    \begin{equation*}
      \cat{BiAlg}
      \quad\text{or}\quad
      \cat{HAlg}
      \quad
      \xrightarrow{\quad\text{forget}\quad}
      \quad
      \cat{Alg}
    \end{equation*}
    for the forgetful functors $\cat{Coalg}\to\cat{Vect}$. $(-)_{\Bbbk'}$ is left adjoint at the level of algebras, all categories in sight are locally presentable \cite[Lemma 2 and Theorem 6]{zbMATH06696043}, etc.

  \item {\bf} The two unlabeled implications in
    \begin{equation*}
      \begin{tikzpicture}[>=stealth,auto,baseline=(current  bounding  box.center)]
        \path[anchor=base] 
        (0,0) node (c) {\Cref{item:th:finextradj:coalgcont}}
        +(2,.5) node (d) {\Cref{item:th:finextradj:coalgpow}}
        +(2,-.5) node (f) {\Cref{item:th:finextradj:bhalgcont}}
        +(4,0) node (g) {\Cref{item:th:finextradj:bhalgpow}}
        +(6,0) node (h) {\Cref{item:th:finextradj:bhalgspecpow}}
        ;
        \draw[-implies,double equal sign distance] (c) to[bend left=6] node[pos=.5,auto] {$\scriptstyle \text{obvious}$} (d);
        \draw[-implies,double equal sign distance] (d) to[bend left=6] node[pos=.5,auto] {$\scriptstyle $} (g);
        \draw[-implies,double equal sign distance] (c) to[bend right=6] node[pos=.5,auto,swap] {$\scriptstyle $} (f);
        \draw[-implies,double equal sign distance] (f) to[bend right=6] node[pos=.5,auto,swap] {$\scriptstyle \text{obvious}$} (g);
        \draw[-implies,double equal sign distance] (g) to[bend right=0] node[pos=.5,auto] {$\scriptstyle \text{obvious}$} (h);
      \end{tikzpicture}
    \end{equation*}
    hold because products (and more generally, limits) of bialgebras or Hopf algebras are computed at the coalgebra level, by the continuity (in fact, {\it monadicity} \cite[Definition 4.4.1]{brcx_hndbk-2}, per \cite[\S 4.1, Proposition]{zbMATH05312006} and \cite[Theorem 10 2.]{zbMATH06696043}) of the corresponding forgetful functors.

  \item {\bf \Cref{item:th:finextradj:fin} $\Rightarrow$ \Cref{item:th:finextradj:coalgcont}:} We break up the construction of a coalgebra product $C:=\prod_i C_i$ into a handful of steps, each of which is in turn preserved by extending scalars along {\it finite} field extensions. The recipe:    
    \begin{itemize}[wide]
    \item First form the vector-space product $W:=\displaystyle \prod^{\cat{Vect}}C_i$. The construction is indeed preserved by $(-)_{\Bbbk'}$ by the assumed finiteness. 

    \item Then form the cofree coalgebra (\Cref{ex:scalewrong}) $C_{\cat{cf}}(W)$, naturally equipped with maps
      \begin{equation*}
        \begin{tikzpicture}[>=stealth,auto,baseline=(current  bounding  box.center)]
          \path[anchor=base] 
          (0,0) node (l) {$C_{\cat{cf}}(W)$}
          +(2,.5) node (u) {$W$}
          +(4,0) node (r) {$C_i$}
          ;
          \draw[->>] (l) to[bend left=6] node[pos=.5,auto] {$\scriptstyle \text{cofree structure map}$} (u);
          \draw[->>] (u) to[bend left=6] node[pos=.5,auto] {$\scriptstyle \text{$i^{th}$-factor projection}$} (r);
          \draw[->>] (l) to[bend right=6] node[pos=.5,auto,swap] {$\scriptstyle \pi_i$} (r);
        \end{tikzpicture}
      \end{equation*}
      (linear only; not, in general, coalgebra morphisms). We relegate the invariance of $W\mapsto C_{\cat{cf}}(W)$ under $(-)_{\Bbbk'}$ to \Cref{pr:finextcofree} below. 

    \item Finally, $C\le C_{\cat{cf}}(W)$ will be the subspace defined as the common equalizer of all parallel pairs
      \begin{equation}\label{eq:cwcwcc}
        \begin{tikzpicture}[>=stealth,auto,baseline=(current  bounding  box.center)]
          \path[anchor=base] 
          (0,0) node (l) {$C_{\cat{cf}}(W)$}
          +(3,.5) node (u) {$C_{\cat{cf}}(W)^{\otimes (n+1)}$}
          +(3,-.5) node (d) {$C_i$}
          +(6,0) node (r) {$C_i^{\otimes (n+1)}$}
          ;
          \draw[->] (l) to[bend left=6] node[pos=.5,auto] {$\scriptstyle \Delta^{(n)}$} (u);
          \draw[->] (u) to[bend left=6] node[pos=.5,auto] {$\scriptstyle \pi_i^{\otimes (n+1)}$} (r);
          \draw[->] (l) to[bend right=6] node[pos=.5,auto,swap] {$\scriptstyle \pi_i$} (d);
          \draw[->] (d) to[bend right=6] node[pos=.5,auto,swap] {$\scriptstyle \Delta^{(n)}$} (r);
        \end{tikzpicture}
      \end{equation}
      where $\Delta^{(\bullet)}$ denote the iterates of the comultiplication. Tensoring with {\it any} overfield preserves both equalizers and arbitrary subspace intersections, finishing the proof of the present implication.
    \end{itemize}    
    
  \item {\bf \Cref{item:th:finextradj:bhalgspecpow} $\Rightarrow$ \Cref{item:th:finextradj:fin}:} The idea, in its essence, is that driving \Cref{ex:scalewrong}. Denote cofree Hopf algebras by $H_{\cat{cf}}(\bullet)$, as in \Cref{le:hcfsurj}. Given the surjectivity, by the latter, of the map 
    \begin{equation*}
      H_{\cat{cf}}(\Bbbk^{S})
      \xrightarrow{\quad}
      \Bbbk^{S}
      ,\quad
      \text{arbitrary set }S
    \end{equation*}
    (for {\it arbitrary} $\Bbbk$, hence also for $\Bbbk'$), the argument of \Cref{ex:scalewrong} carries over: if $S$ is infinite then
    \begin{equation*}
      \left(\Bbbk^S\right)_{\Bbbk'}
      \lhook\joinrel\xrightarrow{\quad}
      \left(\Bbbk'\right)^S
    \end{equation*}
    is not onto. 
  \end{enumerate}
  This suffices to connect the implication graph.
\end{proof}

The following simple remark is perhaps worth setting out on its own; it is auxiliary to the proof of \Cref{th:finextradj} above (implication \Cref{item:th:finextradj:fin} $\Rightarrow$ \Cref{item:th:finextradj:coalgcont}). 

\begin{proposition}\label{pr:finextcofree}
  The scalar extension functor $(-)_{\Bbbk'}$ along a finite field extension $\Bbbk\le \Bbbk'$ intertwines all of the following right adjoints:

  \begin{enumerate}[(1),wide]
  \item\label{item:pr:finextcofree:c2v} the cofree coalgebra $C_{\cat{cf}}(V)\to V$ on a vector space;

  \item\label{item:pr:finextcofree:b2a} the cofree bialgebra $B_{\cat{cf}}(A)\to A$ on an algebra;

  \item\label{item:pr:finextcofree:h2b} the cofree Hopf algebra $H_{\cat{cf}}(B)\to B$ on a bialgebra;

  \item\label{item:pr:finextcofree:h2a} the cofree Hopf algebra $H_{\cat{cf}}(A)\to A$ on an algebra. 
  \end{enumerate}

  of \cite[\S 1.6]{dnr} over $\Bbbk$ and $\Bbbk'$ respectively. 
\end{proposition}
\begin{proof}
  \begin{enumerate}[label={},wide]
  \item {\bf \Cref{item:pr:finextcofree:c2v}:} First, the finiteness of $\Bbbk\le \Bbbk'$ ensures that $(-)_{\Bbbk'}$ intertwines the formation of vector-space duals:
    \begin{equation}\label{eq:candual}
      (V^*)_{\Bbbk'}
      \xrightarrow[\cong]{\quad\text{canonically}\quad}
      (V_{\Bbbk'})^*
      ,\quad
      \forall~ V\in \tensor*[_{\Bbbk}]{\cat{Vect}}{}. 
    \end{equation}
    Secondly, it also preserves the {\it finite dual} construction
    \begin{equation*}
      \begin{aligned}
        \tensor*[_{\Bbbk}]{\cat{Alg}}{}
        \ni
        A
        \xmapsto{\quad}
        A^{\circ}
        &:=
          \left\{f\in A^*\ |\ \ker f\text{ contains a cofinite-dimensional ideal}\right\}\\
        &=
          \bigcup_{A\twoheadrightarrow\overline{A}}\overline{A}^*
          \le A^*
          \quad
          \left(\text{ranging over finite-dimensional quotient algebras }\overline{A}\right)
      \end{aligned}      
    \end{equation*}
    of \cite[\S 1.5]{dnr}, in the sense that \Cref{eq:candual} holds with $A\in \cat{Alg}$ and $(-)^{\circ}$ in place of $V$ and $(-)^*$ respectively. It follows in particular that the scalar extension functor intertwines the cofree coalgebra construction
    \begin{equation*}
      \tensor*[]{\cat{Vect}}{}
      \ni
      V
      \xmapsto{\quad}
      C_{\cat{cf}}(V^{**})
      \xrightarrow[\cong]{\text{\cite[Lemma 1.6.4]{dnr}}}
      T(V^*)^{\circ}
      ,\quad
      T(-):=\text{tensor algebra}. 
    \end{equation*}
    To conclude, recall the following description of $C_{\cat{cf}}(V)$; it can be extracted from the proof of \cite[Theorem 1.6.6]{dnr} (via \cite[Lemma 1.6.5]{dnr}):
    \begin{equation}\label{eq:cvdesc}
      C_{\cat{cf}}(V)
      =
      \left\{
        x\in C_{\cat{cf}}(V^{**})
        \ |\
        \left(\pi_{V^{**}}^{\otimes (n+1)}\circ \Delta^{(n)}\right)(x)
        \in
        V^{\otimes (n+1)}
        \subset
        (V^{**})^{\otimes (n+1)}
      \right\},
    \end{equation}
    where $C_{\cat{cf}}(\bullet)\xrightarrow{\pi_{\bullet}}\bullet$ denotes the structure map of the cofree coalgebra and, as in \Cref{eq:cwcwcc}, $\Delta^{(n)}$ are the iterates of $\Delta$. Everything on the right-hand side of \Cref{eq:cvdesc} is preserved by $(-)_{\Bbbk'}$ (even for {\it arbitrary} field extensions; finiteness is not needed here), hence the conclusion.

  \item {\bf \Cref{item:pr:finextcofree:b2a}:} As observed in \cite[\S VI, discussion following Theorem 6.4.8]{swe}, $B_{\cat{cf}}(A)$ is nothing but the cofree coalgebra $C_{\cat{cf}}(A)$: the algebra structure on $A$ lifts automatically along the canonical map $C_{\cat{cf}}(A)\to A$. 
    
  \item {\bf \Cref{item:pr:finextcofree:h2b}:} This follows from the description of the cofree Hopf algebra $H_{\cat{cf}}(B)\to B$ on a bialgebra $B$ given in \cite[Theorem 3.1(a)]{zbMATH05696924}:
    \begin{itemize}[wide]
    \item Form the product
      \begin{equation*}
        P
        :=
        \prod^{\cat{BiAlg}}_{n\in \bZ_{\ge 0}}B_n
        \cong
        \prod^{\cat{Coalg}}_{n\in \bZ_{\ge 0}}B_n
        ,\quad
        B_n
        :=
        \begin{cases}
          B&\text{if $n$ is even}\\
          B^{op,cop}&\text{otherwise},
        \end{cases}
      \end{equation*}
      where `op' superscripts indicate opposite algebra structures and dually (per standard notation \cite[Definition 2.1.4]{rad}), `cop' means opposite coalgebra structures. That product construction is compatible with $(-)_{\Bbbk'}$ by \Cref{th:finextradj}.

    \item Consider the bialgebra morphism $P^{op,cop}\xrightarrow{S}P$ that ``shifts'' factor indices, defined factor-wise by $B_n\xrightarrow{\id}B_{n-1}$.

    \item Finally, $H_{\cat{cf}}(B)\subseteq P$ is then defined as the largest subcoalgebra contained in
      \begin{equation*}
        \left\{x\in P\ |\ x_1S(x_2) = \eta\varepsilon(x) = S(x_1)x_2\right\}
        ,\quad
        \begin{aligned}
          \varepsilon,\eta
          &:=
            \text{(co)unit of the bialgebra }P,\\
          x
          &\xmapsto[\text{{\it Sweedler notation} \cite[p.23]{rad}}]{\text{comultiplication }\Delta}
            x_1\otimes x_2
        \end{aligned}        
      \end{equation*}
      (automatically a sub-bialgebra, Hopf with antipode $S$). 
    \end{itemize}
    This last step thus first forms an equalizer (in $\cat{vect}$) of the three self-maps
    \begin{equation*}
      \begin{tikzpicture}[>=stealth,auto,baseline=(current  bounding  box.center)]
        \path[anchor=base] 
        (0,0) node (l) {$P$}
        +(4,0) node (r) {$P$}
        ;
        \draw[->] (l) to[bend left=40] node[pos=.5,auto] {$\scriptstyle \text{mult}\circ(\id\otimes S)\circ\Delta$} (r);
        \draw[->] (l) to[bend left=0] node[pos=.5,auto] {$\scriptstyle \eta\circ\varepsilon$} (r);
        \draw[->] (l) to[bend right=40] node[pos=.5,auto,swap] {$\scriptstyle \text{mult}\circ(S\otimes\id)\circ\Delta$} (r);
      \end{tikzpicture}
    \end{equation*}
    and then cuts out the largest subcoalgebra of $P$ contained therein. Both of those procedures are preserved by {\it arbitrary} (not just finite) field extensions: the former obviously, and the latter by \cite[Lemma 4.2]{MR4712418}. 
    
  \item {\bf \Cref{item:pr:finextcofree:h2a}:} This right adjoint is the composition of the preceding two:
    \begin{equation*}
      \begin{tikzpicture}[>=stealth,auto,baseline=(current  bounding  box.center)]
        \path[anchor=base] 
        (0,0) node (l) {$\cat{Alg}$}
        +(2,.5) node (u) {$\cat{BiAlg}$}
        +(4,0) node (r) {$\cat{HAlg}$}
        +(2,.15) node () {\rotatebox[origin=c]{-90}{$\cong$}}
        ;
        \draw[->] (l) to[bend left=6] node[pos=.5,auto] {$\scriptstyle B_{\cat{cf}}$} (u);
        \draw[->] (u) to[bend left=6] node[pos=.5,auto] {$\scriptstyle H_{\cat{cf}}$} (r);
        \draw[->] (l) to[bend right=6] node[pos=.5,auto,swap] {$\scriptstyle H_{\cat{cf}}$} (r);
      \end{tikzpicture}
    \end{equation*}
    \Cref{item:pr:finextcofree:h2a} thus follows from \Cref{item:pr:finextcofree:b2a} + \Cref{item:pr:finextcofree:h2b}. 
  \end{enumerate}
\end{proof}

The construction of $C_{\cat{cf}}(V)$ just recalled in the proof of \Cref{pr:finextcofree}, via the embedding $V\le V^{**}$, can in a sense be turned around: $C_{\cat{cf}}$ can also be built {\it up} (rather than down) from

\begin{lemma}\label{le:cfbuildup}
  The cofree coalgebra construction
  \begin{equation*}
    \cat{Vect}
    \ni
    V
    \xmapsto{\quad}
    C_{\cat{cf}}(V)
    \in
    \cat{Coalg}
  \end{equation*}
  preserves filtered colimits. In particular, we have
  \begin{equation}\label{eq:cfcolim}
    \varinjlim_{\text{fin-dim }V'\le V} C_{\cat{cf}}(V')
    \cong
    \varinjlim_{\text{fin-dim }V'\le V} T(V'^*)^{\circ}
    \xrightarrow[\cong]{\quad\text{canonical map}\quad}
    C_{\cat{cf}}(V).
  \end{equation}
\end{lemma}
\begin{proof}
  This is an application of a familiar general principle applicable to adjunctions between locally presentable categories.

  In both $\cat{Coalg}$ and $\cat{Vect}$ every object is a {\it filtered colimit} \cite[Definition 1.4]{ar} of finite-dimensional subobjects: obvious for vector spaces, and standard \cite[Theorem 5.1.1]{mont} for coalgebras. These, in turn, are precisely the {\it finitely presentable} \cite[Definition 1.1]{ar} objects $\bullet$ in each respective category: their corresponding representable (covariant) functors $\cC(\bullet,-)$ preserve filtered colimits.

  The forgetful functor $\cat{Coalg}\xrightarrow{\cat{fgt}}\cat{Vect}$ is thus a left adjoint between locally finitely presentable categories which respects object finite presentability; per \cite[Theorem 1.66, part (2) of the proof]{ar}, its right adjoint $C_{\cat{cf}}$ respects filtered colimits. 
\end{proof}


The finiteness requirement in \Cref{pr:finextcofree} is also, it turns out, not accidental, but it is sub-optimal: that statement cannot quite be reversed. In fact, preservation by $(-)_{\Bbbk'}$ of the functor $C_{\cat{cf}}$ singles out precisely the {\it algebraic} field extensions. 

\begin{theorem}\label{th:extprescf}
  The following conditions on a field extension $\Bbbk\le \Bbbk'$ are equivalent.
  \begin{enumerate}[(a),wide]
  \item\label{item:th:extprescf:alg} The extension is algebraic.

  \item\label{item:th:extprescf:cbadj} The corresponding functor $(-)_{\Bbbk'}$ intertwines either (equivalently, both) of the adjunctions
    \begin{equation*}
      \begin{tikzpicture}[>=stealth,auto,baseline=(current  bounding  box.center)]
        \path[anchor=base] 
        (0,0) node (vect) {$\cat{Vect}$}
        +(-2,.5) node (coalg) {$\cat{Coalg}$}
        +(4,0) node (alg) {$\cat{Alg}$}
        +(2,.5) node (bialg) {$\cat{BiAlg}$}
        +(-1,.25) node (vc) {\rotatebox[origin=c]{170}{$\top$}}
        +(3,.25) node (ab) {\rotatebox[origin=c]{170}{$\top$}}
        ;
        \draw[->] (vect) to[bend left=20] node[pos=.5,auto] {$\scriptstyle $} (coalg);
        \draw[->] (coalg) to[bend left=20] node[pos=.5,auto] {$\scriptstyle $} (vect);
        \draw[->] (alg) to[bend left=20] node[pos=.5,auto] {$\scriptstyle $} (bialg);
        \draw[->] (bialg) to[bend left=20] node[pos=.5,auto] {$\scriptstyle $} (alg);
      \end{tikzpicture}
    \end{equation*}
    
  \item\label{item:th:extprescf:1d} The functor $(-)_{\Bbbk'}$ preserves the cofree coalgebra on the 1-dimensional vector space (equivalently, cofree bialgebra on the 1-dimensional algebra).

  \item\label{item:th:extprescf:h2alg} The functor $(-)_{\Bbbk'}$ intertwines the adjunction
    \begin{equation*}
      \begin{tikzpicture}[>=stealth,auto,baseline=(current  bounding  box.center)]
        \path[anchor=base] 
        (0,0) node (alg) {$\cat{Alg}$}
        +(-2,.5) node (halg) {$\cat{HAlg}$}
        +(-1,.25) node (ah) {\rotatebox[origin=c]{170}{$\top$}}
        ;
        \draw[->] (alg) to[bend left=20] node[pos=.5,auto] {$\scriptstyle $} (halg);
        \draw[->] (halg) to[bend left=20] node[pos=.5,auto] {$\scriptstyle $} (alg);
      \end{tikzpicture}
    \end{equation*}

    
    
  \item\label{item:th:extprescf:1dalg} $(-)_{\Bbbk'}$ preserves the cofree Hopf algebra on the 1-dimensional algebra.
  \end{enumerate}
  
  
\end{theorem}
\begin{proof}
  \begin{enumerate}[label={},wide]
  \item {\bf \Cref{item:th:extprescf:alg} $\Rightarrow$ \Cref{item:th:extprescf:cbadj}:} We prove the coalgebra version, for the corresponding statement for bialgebras follows. Indeed, the bialgebra adjunction carries the same underlying coalgebra structure: the cofree coalgebra on an algebra, regarded as a vector space, automatically carries a natural bialgebra structure; this is noted explicitly in \cite[\S VI, post Theorem 6.4.8]{swe}.
    
    \Cref{le:cfbuildup} reduces the problem to finite-dimensional $V$, whereupon
    \begin{equation*}      
      C_{\cat{cf}}(V)
      \cong
      C_{\cat{cf}}(V^{**})
      \xrightarrow[\quad\cong\quad]{\text{\cite[Theorem 6.4.1]{swe}}}
      T(V^*)^{\circ}
      \cong
      \Bbbk\braket{\cS}^{\circ},
    \end{equation*}
    $\Bbbk\braket{\cS}$ denoting the free $\Bbbk$-algebra on a basis $S$ for $V^*$.
    
    The functor $(-)_{\Bbbk'}$ (for any $\Bbbk\le \Bbbk'$ at all, algebraic or not) certainly intertwines the formation of duals for finite-dimensional $V$ as well as $T(-)$, so the crux of the matter is the preservation of finite duals, specifically for free finitely-generated algebras $\Bbbk\braket{\cS}$:
    \begin{equation*}
      \Bbbk\le \Bbbk'\text{ algebraic}
      \xRightarrow{\quad}
      \left(\Bbbk\braket{\cS}^{\circ}\right)_{\Bbbk'}
      \xrightarrowdbl[\text{is onto}]{\quad\text{canonical map}\quad}
      \Bbbk'\braket{\cS}^{\circ}
      ,\quad
      \forall S\text{ finite}. 
    \end{equation*}
    This, though, follows from \Cref{le:presfindual} below (implication \Cref{item:le:presfindual:alg} $\Rightarrow$ \Cref{item:le:presfindual:aff}). 
    
  \item {\bf \Cref{item:th:extprescf:cbadj} $\Rightarrow$ \Cref{item:th:extprescf:1d} and \Cref{item:th:extprescf:h2alg} $\Rightarrow$ \Cref{item:th:extprescf:1dalg}} are formal. 
    
  \item {\bf \Cref{item:th:extprescf:alg} $\Rightarrow$ \Cref{item:th:extprescf:h2alg}:} Recall \cite[Definition 4.2.3]{mont} that finitely-generated $\Bbbk$-algebras are sometimes also termed {\it ($\Bbbk$-)affine}. By analogy, we call a Hopf algebra {\it ($\Bbbk$-)S-affine} or {\it ($\Bbbk$-)-antipode-affine} if it is generated as an algebra by the iterates $S^{\bullet}F$ under the antipode of a finite subset (or subspace) $F$. We observe that an arbitrary Hopf algebra $H$ can be recovered as
    \begin{equation}\label{eq:colimsaff}
      H=\varinjlim^{\cat{Alg}}_{\text{S-affine }H'\le H}H'.
    \end{equation}
    What is more, in the present context of a given algebraic extension $\Bbbk\le \Bbbk'$, an S-affine Hopf $\Bbbk'$-algebra $H$ will be of the form $H\cong K_{\Bbbk'}$ for some finite intermediate $\Bbbk(H)/\Bbbk$ and S-affine Hopf $\Bbbk(H)$-algebra $K\le H$: fix a finite-dimensional $\Bbbk'$-coalgebra $C$ generating $H$ as an algebra, take $\Bbbk(H)$ large enough to contain the {\it structure constants} $c_{ij}^k$ in
    \begin{equation*}
      e_i\cdot e_j = \sum_k c_{ij}^k e_k
      ,\quad
      (e_i)\text{ a $\Bbbk'$-basis for the algebra }C^*,
    \end{equation*}
    and take for $K$ the $\Bbbk''$-algebra generated by the antipode-iterates of the dual basis $(e^*_i)$.

    Consider a Hopf $\Bbbk'$-algebra $H$ and an $\Bbbk$-algebra $A$. Relegating one crucial step to \Cref{le:factthroughfinext}, the proof functions as follows:
    \begin{equation*}
      \begin{aligned}
        \tensor*[_{\Bbbk'}]{\cat{Alg}}{}(H,A_{\Bbbk'})
        &\cong
          \varprojlim_{\text{S-affine }H'\le H}\tensor*[_{\Bbbk'}]{\cat{Alg}}{}(H',A_{\Bbbk'})
          \quad\text{by \Cref{eq:colimsaff}}\\
        &\cong
          \varprojlim_{H'}\tensor*[_{\Bbbk'}]{\cat{Alg}}{}(K_{\Bbbk'},A_{\Bbbk'})
          \quad\text{by the preceding discussion}\\
        &\cong
          \varprojlim_{H'}\tensor*[_{\Bbbk(H')}]{\cat{Alg}}{}(K,A_{\Bbbk'})\\
        &\cong
          \varprojlim_{\substack{K\le H\\K\text{ S-affine}/\Bbbk_0\\\Bbbk\le \Bbbk_0\le \Bbbk'\\
        [\Bbbk_0:\Bbbk]<\infty
        }}
        \tensor*[_{\Bbbk_0}]{\cat{Alg}}{}(K,A_{\Bbbk'})\\
        &\cong
          \varprojlim_{\Bbbk_0,K}
          \varinjlim_{\substack{\Bbbk_0\le \Bbbk_1\le \Bbbk'\\
        [\Bbbk_1:\Bbbk_0]<\infty}}
        \tensor*[_{\Bbbk_0}]{\cat{Alg}}{}(K,A_{\Bbbk_1})
        \quad\text{by \Cref{le:factthroughfinext}}\\
        &\cong
          \varprojlim_{\Bbbk_0,K}
          \varinjlim_{\Bbbk_1}
          \tensor*[_{\Bbbk_0}]{\cat{HAlg}}{}(K,H_{\cat{cf}}(A_{\Bbbk_1}))
        \quad(\text{$H_{\cat{cf}}$ universality})\\
        &\cong
          \varprojlim_{\Bbbk_0,K}
          \varinjlim_{\Bbbk_1}
          \tensor*[_{\Bbbk_0}]{\cat{HAlg}}{}(K,H_{\cat{cf}}(A)_{\Bbbk_1})
          \quad\text{by \Cref{pr:finextcofree}\Cref{item:pr:finextcofree:h2a}}\\
        &\cong
          \varprojlim_{\Bbbk_0,K}
          \tensor*[_{\Bbbk_0}]{\cat{HAlg}}{}\left(K,\varinjlim_{\Bbbk_1} H_{\cat{cf}}(A)_{\Bbbk_1}\right)
        \quad\text{also \Cref{le:factthroughfinext}}\\
        &\cong
          \varprojlim_{\Bbbk_0,K}
          \tensor*[_{\Bbbk_0}]{\cat{HAlg}}{}\left(K,H_{\cat{cf}}(A)_{\Bbbk'}\right)\\
        &\cong
          \tensor*[_{\Bbbk'}]{\cat{HAlg}}{}(H,H_{\cat{cf}}(A)_{\Bbbk'}).
      \end{aligned}      
    \end{equation*}

  \item {\bf \Cref{item:th:extprescf:1d} $\Rightarrow$ \Cref{item:th:extprescf:alg}:} Since $(-)_{\Bbbk'}$ in any case respects duality for finite-dimensional vector spaces (such as the $V\cong \Bbbk$ under consideration), the issue is whether it respects the finite duality in $C_{\cat{cf}}(V)\cong \Bbbk[x]^{\circ}$ (for $x$ spanning $V^*$). For that, we can fall back on the implication \Cref{item:le:presfindual:kx} $\Rightarrow$ \Cref{item:le:presfindual:alg} of \Cref{le:presfindual}.
    
  \item {\bf \Cref{item:th:extprescf:1dalg} $\Rightarrow$ \Cref{item:th:extprescf:alg}:} The selfsame implication \Cref{item:le:presfindual:kx} $\Rightarrow$ \Cref{item:le:presfindual:alg} of \Cref{le:presfindual} also handles this claim, given that $H_{\cat{cf}}(\Bbbk)\cong \Bbbk[x^{\pm 1}]^{\circ}$ (\Cref{res:1varrec}\Cref{item:res:1varrec:s1}).
  \end{enumerate}
\end{proof}

\begin{lemma}\label{le:factthroughfinext}
  Let $H$ be an S-affine Hopf $\Bbbk$-algebra, $A$ a $\Bbbk$-algebra, and $\Bbbk\le \Bbbk'$ a field extension. An algebra morphism $H\to A_{\Bbbk'}$ factors through $A_{\Bbbk''}$ for some intermediate $\Bbbk\le \Bbbk''\le \Bbbk'$  finitely-generated over $\Bbbk$. 
\end{lemma}
\begin{proof}
  The assumption is that $H$ is generated as an algebra by the iterates $S^{\bullet C}$ for some finite-dimensional coalgebra $C\le H$, $S$ being the antipode of $H$. Coalgebra morphisms $C\to H$ factor as
  \begin{equation*}
    \begin{tikzpicture}[>=stealth,auto,baseline=(current  bounding  box.center)]
      \path[anchor=base] 
      (0,0) node (l) {$C$}
      +(2,.5) node (u) {$H_{\cat{f}}(C)$}
      +(4,0) node (r) {$H$}
      ;
      \draw[->] (l) to[bend left=6] node[pos=.5,auto] {$\scriptstyle $} (u);
      \draw[->] (u) to[bend left=6] node[pos=.5,auto] {$\scriptstyle \cat{HAlg}\text{ morphism}$} (r);
      \draw[->] (l) to[bend right=6] node[pos=.5,auto,swap] {$\scriptstyle $} (r);
    \end{tikzpicture}
  \end{equation*}
  through the {\it free Hopf algebra} \cite[\S 1]{zbMATH03344702} $H_{\cat{f}}(C)$ on $C$. There is no loss in assuming $H=H_{\cat{f}}(C)$, which we henceforth do. 

  $\Bbbk$-algebra morphisms $H_{\cat{f}}(C)\to A_{\Bbbk'}$ are in bijection \cite[Corollary 10]{zbMATH03344702} with the sequences
    \begin{equation*}
      (x_n)_{n\in \bZ_{\ge 0}}
      \subset
      C^*\otimes A_{\Bbbk'}
      ,\quad
      x_{n+1}
      =
      \begin{cases}
        \text{the inverse }x_n^{-1}\text{ in }C^*\otimes A_{\Bbbk'}&\text{if $n$ is even}\\
        x_n^{-1}\text{ in }C^*\otimes(A_{\Bbbk'})^{op}&\text{otherwise}
      \end{cases}
    \end{equation*}
    The first term $x_0$ of that sequence belongs to some $A_{\Bbbk''}$ for finitely-generated intermediate $\Bbbk''$ as in the statement. That the inverse $x_1=x_0^{-1}$ then also belongs to $A_{\Bbbk''}$ (as do the other $x_n$ afterwards, inductively) follows: multiplication by $x_0$, regarded as a $\Bbbk''$-linear endomorphism, preserves the direct-sum decomposition
    \begin{equation*}
      A_{\Bbbk'}\cong A_{\Bbbk''}\oplus \left(A\otimes_{\Bbbk}\Bbbk''^{\perp}\right)
      \text{ in }
      \tensor*[_{\Bbbk''}]{\cat{Vect}}{}
      ,\quad
      \Bbbk' = \Bbbk''\oplus \Bbbk''^{\perp}      
    \end{equation*}
    for some complementary subspace $\Bbbk''^{\perp}$ of $\Bbbk''\le \Bbbk'$ (i.e. leaves invariant both summands); its inverse must do so as well.
\end{proof}

\begin{lemma}\label{le:presfindual}
  The following conditions on a field extension $\Bbbk\le \Bbbk'$ are equivalent.
  \begin{enumerate}[(a),wide]
  \item\label{item:le:presfindual:alg} The extension is algebraic.

  \item\label{item:le:presfindual:aff} $(-)_{\Bbbk'}$ preserves the formation of finite duals for affine algebras.

  \item\label{item:le:presfindual:kx} $(-)_{\Bbbk'}$ preserves either one of the finite duals $\Bbbk[x]^{\circ}$ or $\Bbbk[x^{\pm 1}]^{\circ}$.
  \end{enumerate}
\end{lemma}
\begin{proof}
  \begin{enumerate}[label={},wide]

  \item {\bf \Cref{item:le:presfindual:alg} $\Rightarrow$ \Cref{item:le:presfindual:aff}:} Let $A$ be an affine $\Bbbk$-algebra. Recast, the claim is that any cofinite ideal of $A_{\Bbbk'}$ contains one such ideal for $A\le A_{\Bbbk'}$ (provided $\Bbbk\le \Bbbk'$ is algebraic). Or again: for a finite-dimensional $\Bbbk'$-algebra $B$, a finite subset $F\subseteq B$ generates a finite-dimensional $\Bbbk$-algebra.
    
    In this last formulation, simply note that finite dimensionality means that for some $\ell$ all words of length $>\ell$ in $F$ are $\Bbbk'$-linear combinations of words of length $\le \ell$. The coefficients appearing in such expansions for length-$(\ell+1)$ words are all contained in some {\it finite} intermediate extension $\Bbbk\le \widetilde{\Bbbk}\le \Bbbk'$, so the $\widetilde{\Bbbk}$-algebra $\widetilde{\Bbbk}(F)\le B$ generated by $F$ is certainly finite-dimensional over $\widetilde{\Bbbk}$. But then it is also finite-dimensional over $\Bbbk$, and along with it its $\Bbbk$-subspace $\Bbbk(F)\le \widetilde{\Bbbk}(F)$. 
    
  \item {\bf \Cref{item:le:presfindual:aff} $\Rightarrow$ \Cref{item:le:presfindual:kx}} is self-evident.

  \item {\bf \Cref{item:le:presfindual:kx} $\Rightarrow$ \Cref{item:le:presfindual:alg}:} Consider $t\in \Bbbk'$ transcendental over $\Bbbk$ and observe that the algebra morphism
    \begin{equation*}
      \Bbbk'[x]\cong \Bbbk[x]_{\Bbbk'}
      \ni
      x
      \xmapsto{\quad}
      t
      \in
      \Bbbk'
    \end{equation*}
    does not annihilate any cofinite ideals in $\Bbbk[x]$. For that reason, regarded as an element of $\Bbbk'[x]^{\circ}$, that character cannot belong to $(\Bbbk[x]^{\circ})_{\Bbbk'}$. The argument works just as well for $\Bbbk[x^{\pm 1}]^{\circ}$: pick an {\it invertible} transcendental element $t\in \Bbbk'$.
  \end{enumerate}
\end{proof}

Affineness is essential in \Cref{le:presfindual}, in the proof of the implication \Cref{item:le:presfindual:alg} $\Rightarrow$ \Cref{item:le:presfindual:aff}:

\begin{example}\label{ex:infncpoly}
  No infinite extension $\Bbbk\le \Bbbk'$ (algebraic or not) preserves the formation of finite duals for (commutative or non-commutative) polynomial rings on infinite sets $S$: an algebra morphism $\Bbbk'[S]\to \Bbbk'$ for which the $\Bbbk$-span of $S$ is infinite over $\Bbbk$ cannot annihilate a cofinite ideal of $\Bbbk[S]$. 
\end{example}

A little additional work will bring us full circle back to products, in the following analogue of \Cref{th:finextradj} characterizing {\it algebraic} (as opposed to finite) field extensions.

\begin{theorem}\label{th:algextradj}
  The following conditions on a field extension $\Bbbk\le \Bbbk'$ are equivalent.
  \begin{enumerate}[(a),wide]
  \item\label{item:th:algextradj:fin} The extension is algebraic. 

  \item\label{item:th:algextradj:fdcoalgcont} \Cref{eq:scalextcoalg} preserves finite (or binary) products of finite-dimensional coalgebras. 
    
  \item\label{item:th:algextradj:coalgcont} \Cref{eq:scalextcoalg} is finitely continuous or, equivalently, preserves finite (or binary) products.

  \item\label{item:th:algextradj:coalgpow} \Cref{eq:scalextcoalg} preserves finite (or binary) powers.

  \item\label{item:th:algextradj:bhalgcont} The analogue of \Cref{item:th:algextradj:coalgcont} for categories of bialgebras and/or Hopf algebras.

  \item\label{item:th:algextradj:bhalgpow} The analogue of \Cref{item:th:algextradj:coalgpow} for categories of bialgebras and/or Hopf algebras.
    
  \item\label{item:th:algextradj:bhalgspecpow} $(-)_{\Bbbk'}$ preserves the binary power of some Hopf algebra of finite dimension $>1$.
  \item\label{item:th:algextradj:calgspecpow} $(-)_{\Bbbk'}$ preserves the binary power of some coalgebra of finite dimension $>1$.
  \end{enumerate}
\end{theorem}
\begin{proof}
  `Finite' and `binary' are interchangeable throughout by induction, and the conditions in \Cref{item:th:algextradj:coalgcont} are equivalent by \cite[Proposition 13.3]{ahs} given that, as noted previously in the proof of \Cref{th:finextradj}, equalizers are preserved regardless of the field extension. Immediate implications include
  
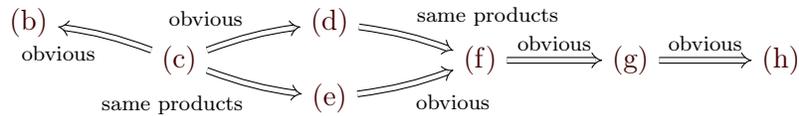
\begin{equation*}
    \begin{tikzpicture}[>=stealth,auto,baseline=(current  bounding  box.center)]
      \path[anchor=base] 
      (0,0) node (fcont) {\Cref{item:th:algextradj:coalgcont}}
      +(2,.5) node (fpow) {\Cref{item:th:algextradj:coalgpow}}
      +(2,-.5) node (fcontbh) {\Cref{item:th:algextradj:bhalgcont}}
      +(4,0) node (fpowbh) {\Cref{item:th:algextradj:bhalgpow}}
      +(6,0) node (2pow1h) {\Cref{item:th:algextradj:bhalgspecpow}}
      +(8,0) node (2pow1c) {\Cref{item:th:algextradj:calgspecpow}}
      +(-2,.5) node (fdcoalg) {\Cref{item:th:algextradj:fdcoalgcont}}
      ;
      \draw[-implies,double equal sign distance] (fcont) to[bend left=6] node[pos=.5,auto] {$\scriptstyle \text{obvious}$} (fpow);
      \draw[-implies,double equal sign distance] (fpow) to[bend left=6] node[pos=.5,auto] {$\scriptstyle \text{same products}$} (fpowbh);
      \draw[-implies,double equal sign distance] (fcont) to[bend right=6] node[pos=.5,auto,swap] {$\scriptstyle \text{same products}$} (fcontbh);
      \draw[-implies,double equal sign distance] (fcontbh) to[bend right=6] node[pos=.5,auto,swap] {$\scriptstyle \text{obvious}$} (fpowbh);
      \draw[-implies,double equal sign distance] (fpowbh) to[bend right=0] node[pos=.5,auto] {$\scriptstyle \text{obvious}$} (2pow1h);
      \draw[-implies,double equal sign distance] (2pow1h) to[bend right=0] node[pos=.5,auto] {$\scriptstyle \text{obvious}$} (2pow1c);
      \draw[-implies,double equal sign distance] (fcont) to[bend right=6] node[pos=.5,auto] {$\scriptstyle \text{obvious}$} (fdcoalg);
    \end{tikzpicture}
  \end{equation*}
  A few others will complete the proof.

  \begin{enumerate}[label={},wide]
  \item {\bf \Cref{item:th:algextradj:fdcoalgcont} $\Rightarrow$ \Cref{item:th:algextradj:coalgcont}} follows from the fact \cite[Theorem 5.1.1]{mont} that every coalgebra $C$ is a filtered colimit $C\cong \varinjlim C'$ of its finite-dimensional subcoalgebras, along with the commutation \cite[Proposition 1.59]{ar}, given the local finite presentability of $\cat{Coalg}$ recalled in the proof of \Cref{le:cfbuildup} of finite limits and filtered colimits:
    \begin{equation*}
      \varinjlim_{\substack{C'_i\le C_i\\\dim C'_i<\infty}}\prod_i^{\text{finite}}C'_i
      \xrightarrow[\cong]{\quad\text{canonical}\quad}
      \prod_i \left(\varinjlim C'_i\right)
      \cong
      \prod_i C_i.
    \end{equation*}

  \item {\bf \Cref{item:th:algextradj:fin} $\Rightarrow$ \Cref{item:th:algextradj:fdcoalgcont}:} For finite-dimensional coalgebras $C$ the unit morphism $C\to C^{*\circ}$ attached to the contravariant right-hand adjunction (\cite[Theorem 1.5.22]{dnr}, \cite[Theorem 6.0.5]{swe})
    \begin{equation*}
      \begin{tikzpicture}[>=stealth,auto,baseline=(current  bounding  box.center)]
        \path[anchor=base] 
        (0,0) node (l) {$\cat{Coalg}$}
        +(4,0) node (r) {$\cat{Alg}$}
        ;
        \draw[->] (l) to[bend left=6] node[pos=.5,auto] {$\scriptstyle *$} (r);
        \draw[->] (r) to[bend left=6] node[pos=.5,auto] {$\scriptstyle \circ$} (l);
      \end{tikzpicture}
    \end{equation*}
    is an isomorphism. Contravariant right adjoints convert coproducts into products, so that
    \begin{equation}\label{eq:prodfdcoalg}
      \dim C_i<\infty
      \xRightarrow{\quad}
      \prod^{\cat{Coalg}}_i C_i
      \cong
      \left(\coprod^{\cat{Alg}}_i C_i^*\right)^{\circ}.
    \end{equation}
    It remains to observe that arbitrary scalar extensions preserve algebra coproducts and duality for finite-dimensional vector spaces, whereas {\it algebraic} scalar extensions preserve finite duals of affine algebras by \Cref{le:presfindual}.
    
  \item {\bf \Cref{item:th:algextradj:calgspecpow} $\Rightarrow$ \Cref{item:th:algextradj:fin}:} \Cref{eq:prodfdcoalg} applies to the binary power of a finite-dimensional $C\in \cat{Coalg}$. The claim is thus that, given
    \begin{equation*}
      (C^*=)A\in \cat{Alg}
      ,\quad
      1<\dim A<\infty,
    \end{equation*}
    preservation by $(-)_{\Bbbk}$ of the finite dual $\left(A\coprod A\right)^{\circ}$ forces $\Bbbk\le \Bbbk'$ to be algebraic. Rephrased, the goal is to show that whenever $\Bbbk\le \Bbbk'$ is transcendental there is a morphism
    \begin{equation}\label{eq:aa2b}
      A\coprod A
      \xrightarrow[\quad\text{in $\tensor*[_{\Bbbk}]{\cat{Alg}}{}$}\quad]{\quad\varphi_{\ell,r}\text{ for `left' and `right'}\quad}
      B_{\Bbbk'}
      \quad\text{for}\quad
      B\in \tensor*[_{\Bbbk}]{\cat{Alg}}{}
      ,\quad
      \dim B<\infty
    \end{equation}
    which does {\it not} factor through a finite-dimensional $\Bbbk$-subalgebra of $B_{\Bbbk'}$. Take for $B$ the matrix algebra $\End(A)\cong M_{\dim A}(\Bbbk)$. The argument bifurcates, referring to a transcendental $t\in \Bbbk'$ fixed throughout.

    \begin{enumerate}[(I),wide]
    \item {\bf $A$ contains, possibly after an appropriate (harmless) finite extension of $\Bbbk$, a non-zero nilpotent element.} It then also contains a square-zero $0\ne x\in A$. Fix a square-zero $0\ne x\in A$. $\varphi_{\ell}$ will simply be the left regular self-action
      \begin{equation*}
        A
        \ni a
        \xmapsto{\quad}
        (a\cdot)
        \in
        B=\End(A),
      \end{equation*}
      whereas $\varphi_r$ will be the same embedding followed by a conjugation in $B_{\Bbbk(t)}\cong M_{\dim A}(\Bbbk(t))$. Choosing that conjugation so that $\varphi_r(x)=x't$ with
      \begin{equation*}
        x'\in B=M_{\dim A}(\Bbbk)
        ,\quad                
        xx'\text{ is not nilpotent},
      \end{equation*}
      the images through the corresponding \Cref{eq:aa2b} of the powers of
      \begin{equation*}
        \left(\text{left-hand copy of }x\right)\cdot \left(\text{right-hand copy of }x\right)\in A\coprod A
      \end{equation*}
      have infinite-dimensional $\Bbbk$-span $\spn\left\{\left(xx'\right)^nt^n\right\}_n$.

    \item {\bf No finite extensions of $\Bbbk$ produce non-zero nilpotent elements in $A$.} We can then assume (again, perhaps after a scalar extension) that $A$ is a finite product of (at least two) copies of $\Bbbk$. Focusing on $\varphi_{\ell,r}$ that factor through some quotient $A\xrightarrowdbl{} \Bbbk^2$, we may as well take $A=\Bbbk^2$; algebra morphisms \Cref{eq:aa2b}, then, simply pick out idempotents $\varphi_{\ell,r}(x)$, $x:=(1,0)\in \Bbbk^2$. Define
      \begin{equation*}
        \begin{tikzpicture}[>=stealth,auto,baseline=(current  bounding  box.center)]
          \path[anchor=base] 
          (0,0) node (l) {$x$}
          +(6,.5) node (u) {$
            \begin{pmatrix}
              1&0\\
              1&0
            \end{pmatrix}
            $}
          +(6,-.5) node (d) {$
            \begin{pmatrix}
              1&t\\
              0&1
            \end{pmatrix}
            \cdot
            \begin{pmatrix}
              1&0\\
              1&0
            \end{pmatrix}
            \cdot
            \begin{pmatrix}
              1&-t\\
              0&\phantom{-}1
            \end{pmatrix}
            =
            \begin{pmatrix}
              1+t&-t+t^2\\
              1&-t
            \end{pmatrix}
            $}
          ;
          \draw[|->] (l) to[bend left=6] node[pos=.5,auto] {$\scriptstyle \varphi_{\ell}$} (u);
          \draw[|->] (l) to[bend right=6] node[pos=.5,auto,swap] {$\scriptstyle \varphi_{r}$} (d);
        \end{tikzpicture}
      \end{equation*}
      We thus have
      \begin{equation*}
        (\text{right-hand $x$})-(\text{left-hand $x$})
        \xmapsto{\quad\text{\Cref{eq:aa2b}}\quad}
        \begin{pmatrix}
          1+t&-t+t^2\\
          1&-t
        \end{pmatrix}
        -
        \begin{pmatrix}
          1&0\\
          1&0
        \end{pmatrix}
        =
        \begin{pmatrix}
          t&-t+t^2\\
          0&-t
        \end{pmatrix},
      \end{equation*}
      an element whose powers are plainly have infinite $\Bbbk$-span.
    \end{enumerate}
  \end{enumerate}
\end{proof}


We record the following description of the cofree Hopf algebra on a finite-dimensional algebra.

\begin{proposition}\label{pr:hcffdalg}
  In the diagram
  \begin{equation*}
    \begin{tikzpicture}[>=stealth,auto,baseline=(current  bounding  box.center)]
      \path[anchor=base] 
      (0,0) node (l) {$H_{\cat{cf}}(A)$}
      +(3,.5) node (u) {$B_{\cat{cf}}(A)\cong T(A^*)^{\circ}$}
      +(6,0) node (r) {$A$,}
      ;
      \draw[->] (l) to[bend left=6] node[pos=.5,auto] {$\scriptstyle $} (u);
      \draw[->] (u) to[bend left=6] node[pos=.5,auto] {$\scriptstyle $} (r);
      \draw[->] (l) to[bend right=6] node[pos=.5,auto,swap] {$\scriptstyle $} (r);
    \end{tikzpicture}
  \end{equation*}
  relating the cofree bialgebra and Hopf algebra on a finite-dimensional $\Bbbk$-algebra $A$ the left-hand arrow is the finite dual
  \begin{equation*}
    H_{\cat{F}}(T(A^*))^{\circ}
    \xrightarrow{\quad}
    T(A^*)^{\circ}
  \end{equation*}
  of the canonical map $T(A^*)\to H_{\cat{F}}(T(A^*))$ of the bialgebra $T(A^*)$ into its Hopf envelope. 
\end{proposition}
\begin{proof}
  The bialgebra $T(A^*)$ in the statement carries the tensor algebra structure and the coalgebra structure induced by that of $A^*$. That the cofree bialgebra is indeed $T(A^*)^{\circ}$ follows from \cite[pp.134-135]{swe}, given that by (the proof of) \cite[Theorem 6.4.1]{swe} this is also the cofree coalgebra on the finite-dimensional vector space $A\cong A^{**}$. 

  Note next that $(-)^{\circ}$ induces a contravariant self-adjoint on the right on both categories $\cat{BiAlg}$ and $\cat{HAlg}$: the natural bijection
  \begin{equation*}
    \cat{BiAlg}(H,K^{\circ})
    \cong
    \cat{BiAlg}(K,H^{\circ})
    ,\quad
    H,K
    \in
    \cat{BiAlg}
    \ \text{or}\ 
    \cat{HAlg}
  \end{equation*}
  follows essentially from \cite[Proposition 7.7.5]{rad}, which identifies both spaces with that of {\it bialgebra (or Hopf) pairings} (\cite[Definition 7.7.6]{rad}, \cite[\S 4.1 D, pp.114-115]{cp_qg}) $H\times K\to \Bbbk$. Said self-adjoint $(-)^{\circ}$ also intertwines the forgetful functor from Hopf algebras to bialgebras, hence the commutativity up to natural isomorphism of (the northwestward branch of)
  \begin{equation*}
    \begin{tikzpicture}[>=stealth,auto,baseline=(current  bounding  box.center)]
      \path[anchor=base] 
      (0,0) node (l) {$\cat{HAlg}$}
      +(2,.5) node (u) {$\cat{HAlg}$}
      +(2,-.5) node (d) {$\cat{BiAlg}$}
      +(4,0) node (r) {$\cat{BiAlg}$}
      +(3,.7) node (ut) {\rotatebox[origin=c]{-10}{$\top$}}
      +(1,-.7) node (dt) {\rotatebox[origin=c]{-10}{$\top$}}
      ;
      \draw[->] (l) to[bend left=6] node[pos=.5,auto] {$\scriptstyle \circ$} (u);
      \draw[->] (u) to[bend left=6] node[pos=.5,auto] {$\scriptstyle $} (r);
      \draw[->] (l) to[bend right=6] node[pos=.5,auto,swap] {$\scriptstyle $} (d);
      \draw[->] (d) to[bend right=6] node[pos=.5,auto,swap] {$\scriptstyle \circ$} (r);
      \draw[->] (r) to[bend right=60] node[pos=.5,auto,swap] {$\scriptstyle H_{\cat{cf}}$} (u);
      \draw[->] (d) to[bend left=60] node[pos=.5,auto] {$\scriptstyle H_{\cat{f}}$} (l);
    \end{tikzpicture}
  \end{equation*}
  More elaborately (but perhaps also more transparently), given \cite[\S IV.8, Theorem 1]{mcl_2e} that the adjoint of a composition of functors is the composition of the corresponding adjoints, we have adjunctions
  \begin{equation*}
    \begin{tikzpicture}[>=stealth,auto,baseline=(current  bounding  box.center)]
      \path[anchor=base] 
      (0,0) node (1b) {$\cat{BiAlg}$}
      +(2,1) node (1h) {$\cat{HAlg}$}
      +(4,0) node (1ho) {$\cat{HAlg}^{op}$}

      +(0,-1) node (2b) {$\cat{BiAlg}$}
      +(2,0) node (2h) {$\cat{HAlg}$}
      +(4,-1) node (2ho) {$\cat{HAlg}^{op}$}

      +(2,-2) node (3bo) {$\cat{BiAlg}^{op}$}

      +(0,-2) node (4b) {$\cat{BiAlg}$}
      +(2,-3) node (4bo) {$\cat{BiAlg}^{op}$}
      +(4,-2) node (4ho) {$\cat{HAlg}^{op}$}
      
      +(1,.2) node () {$\bot$}
      +(3,.2) node () {$\bot$}      

      +(2,-1) node () {\rotatebox[origin=c]{90}{$\cong$}}

      +(1,-2.2) node () {$\top$}
      +(3,-2.2) node () {$\top$}
      ;
      \draw[->] (1b) to[bend left=6] node[pos=.5,auto] {$\scriptstyle H_{\cat{F}}$} (1h);
      \draw[->] (1h) to[bend left=6] node[pos=.5,auto] {$\scriptstyle \circ$} (1ho);

      \draw[<-] (2b) to[bend left=6] node[pos=.5,auto,swap] {$\scriptstyle \cat{fgt}$} (2h);
      \draw[<-] (2h) to[bend left=6] node[pos=.5,auto,swap] {$\scriptstyle \circ$} (2ho);

      \draw[<-] (2b) to[bend right=6] node[pos=.5,auto] {$\scriptstyle \circ$} (3bo);
      \draw[<-] (3bo) to[bend right=6] node[pos=.5,auto] {$\scriptstyle \cat{fgt}$} (2ho);

      \draw[->] (4b) to[bend right=6] node[pos=.5,auto,swap] {$\scriptstyle \circ$} (4bo);
      \draw[->] (4bo) to[bend right=6] node[pos=.5,auto,swap] {$\scriptstyle H_{\cat{cf}}$} (4ho);
    \end{tikzpicture}
  \end{equation*}
  so that the top and bottom must be naturally isomorphic; $(-)^{\circ}$, in other words, intertwines the free and cofree Hopf algebra constructions. This, though, is precisely the sought-after claim in our particular setup: that $(-)^{\circ}$ turns the free Hopf algebra structure map $B\to H_{\cat{F}}(B)$ into the cofree counterpart $H_{\cat{F}}(B)^{\circ}\to B^{\circ}$ thereof with $B=T(A^*)^{\circ}$.
\end{proof}

\begin{remarks}\label{res:freebialg}
  \begin{enumerate}[(1),wide]
  \item\label{item:res:freebialg:twcirc} The principle underlying \Cref{pr:hcffdalg} also fits the isomorphism $C_{\cat{cf}}(V^{**})\cong T(V^*)^{\circ}$ of \cite[Lemma 1.6.4]{dnr} into the broader observation that $C_{\cat{cf}}(W^*)\cong T(W)^{\circ}$: the latter is (the image of $W$ through) the composition
    \begin{equation*}
      \cat{Vect}
      \xrightarrow{\quad T\quad}
      \cat{Alg}
      \xrightarrow{\quad \circ\quad}
      \cat{Coalg}^{op}
    \end{equation*}
    of left adjoints; the former instead chains together
    \begin{equation*}
      \cat{Vect}
      \xrightarrow{\quad *\quad}
      \cat{Vect}^{op}
      \xrightarrow{\quad C_{\cat{cf}}\quad}
      \cat{Coalg}^{op}
      ,
    \end{equation*}
    again left adjoints. The common right adjoint to the two compositions is
    \begin{equation*}
      \begin{tikzpicture}[>=stealth,auto,baseline=(current  bounding  box.center)]
        \path[anchor=base] 
        (0,0) node (l) {$\cat{Vect}$}
        +(2,.5) node (u) {$\cat{Alg}$}
        +(2,-.5) node (d) {$\cat{Vect}^{op}$}
        +(4,0) node (r) {$\cat{Coalg}^{op}$.}
        +(2,0) node () {\rotatebox[origin=c]{90}{$\cong$}}
        ;
        \draw[<-] (l) to[bend left=6] node[pos=.5,auto] {$\scriptstyle \cat{fgt}$} (u);
        \draw[<-] (u) to[bend left=6] node[pos=.5,auto] {$\scriptstyle *$} (r);
        \draw[<-] (l) to[bend right=6] node[pos=.5,auto,swap] {$\scriptstyle *$} (d);
        \draw[<-] (d) to[bend right=6] node[pos=.5,auto,swap] {$\scriptstyle \cat{fgt}$} (r);
      \end{tikzpicture}
    \end{equation*}

  \item\label{item:res:freebialg:hfhf} Dual to the phenomenon recalled in \Cref{re:monadlift}, of the cofree coalgebra $C_{\cat{cf}}(A)$ on an algebra $A$ inheriting a natural (bi)algebra structure, the tensor algebra $T(C)$ on a coalgebra is nothing but the free bialgebra on $C$ (\cite[Theorem 5.3.1]{rad}, \cite[Proposition 3.2.4 and Exercise (1) following it]{swe}): the result of applying to $C$ the left adjoint of $\cat{BiAlg}\to\cat{Coalg}$.
    
    In light of this, the free Hopf algebra $H_{\cat{f}}(T(A^*))$ of \Cref{pr:hcffdalg} is also the free Hopf algebra on the coalgebra $A^*$ (introduced in \cite{zbMATH03344702}; see also \cite[\S 7.5]{rad}): the image of $C$ through chaining together the left adjoints of the two composable functors in
    \begin{equation*}
      \begin{tikzpicture}[>=stealth,auto,baseline=(current  bounding  box.center)]
        \path[anchor=base] 
        (0,0) node (l) {$\cat{HAlg}$}
        +(2,.5) node (u) {$\cat{BiAlg}$}
        +(4,0) node (r) {$\cat{Coalg}$,}
        ;
        \draw[->] (l) to[bend left=6] node[pos=.5,auto] {$\scriptstyle \cat{fgt}$} (u);
        \draw[->] (u) to[bend left=6] node[pos=.5,auto] {$\scriptstyle \cat{fgt}$} (r);
        \draw[->] (l) to[bend right=6] node[pos=.5,auto,swap] {$\scriptstyle \cat{fgt}$} (r);
      \end{tikzpicture}
    \end{equation*}
    producing the left adjoint of the composite. 
  \end{enumerate}
\end{remarks}

An illustration of \Cref{pr:hcffdalg}:

\begin{example}\label{ex:laurent}
  Set $A:=\Bbbk^S$ for a finite set $S$. The coalgebra $A^*$ is then the span of $|S|$ {\it grouplikes} \cite[Definition 1.3.4(a)]{mont} and $T(A^*)\cong \Bbbk\braket{\cS}$, the free algebra on said grouplike elements (with the bialgebra structure resulting therefrom). The free Hopf algebra on $\Bbbk\braket{\cS}$ is easily seen to be what we will denote by $\Bbbk\braket{\cS^{\pm 1}}$: the non-commutative Laurent polynomial algebra, obtained by formally adjoining inverses to every $s\in S$.

  The cofree Hopf algebra map $\Bbbk\braket{\cS^{\pm 1}}^{\circ} \to \Bbbk\braket{\cS}^{\circ}$ of \Cref{pr:hcffdalg} is an embedding: if a functional $\Bbbk\braket{\cS^{\pm 1}}\to \Bbbk$ annihilates all of $\Bbbk\braket{\cS}$ and also a cofinite ideal, it must annihilate arbitrary monomials
  \begin{equation}\label{eq:skek}
    s_1^{e_1}\cdots s_n^{e_n}
    ,\quad
    s_k\in \cS
    ,\quad
    e_k\in \bZ
  \end{equation}
  (and hence vanish identically) by induction on the absolute value of the smallest exponent $e_k<0$. Suppose, say, that $e_k<0$, $f$'s vanishing on a cofinite ideal implies that it annihilates elements of the form
  \begin{equation*}
    \bullet\cdot p(s)\cdot \bullet'
    ,\quad
    \bullet,\bullet'\in \Bbbk{\cS^{\pm 1}}
    \quad\text{and some polynomial}\quad
    p\in\Bbbk[x].
  \end{equation*}
  \Cref{eq:skek} can thus be expressed as a linear combination of analogous terms with strictly larger exponents in place of $e_k$, completing the induction step. 
\end{example}

\begin{remarks}\label{res:1varrec}
  \begin{enumerate}[(1),wide]
  \item\label{item:res:1varrec:stillks} The isomorphism $B_{\cat{cf}}(\Bbbk^\cS)\cong \Bbbk\braket{\cS}^{\circ}$ is valid generally, regardless of whether $\cS$ is finite (as it was in \Cref{ex:laurent}): $\Bbbk^\cS\cong (\Bbbk^{\oplus \cS})^*$, and \Cref{res:freebialg}\Cref{item:res:freebialg:twcirc} applies. 

  \item\label{item:res:1varrec:s1} The case $|\cS|=1$ of \Cref{ex:laurent} features prominently in \cite{zbMATH04204425}. Switching to the more common notation $\Bbbk[x]$ for the single $x\in \cS$, \cite[\S 2]{zbMATH04204425} identifies $\Bbbk[x^{\pm 1}]^{\circ}$ as a Hopf subalgebra of the bialgebra $\Bbbk[x]^{\circ}$ of {\it linearly recursive sequences} \cite[\S 1.1.1]{epsw_rec}: a functional $\Bbbk[x]\xrightarrow{f} \Bbbk$ belongs to the finite dual precisely when the corresponding sequence
  \begin{equation*}
    (f_n)_{n\in \bZ_{\ge 0}}
    ,\quad
    f_n:=f(x^n)
  \end{equation*}
  satisfies a linear recurrence. The Hopf subalgebra $\Bbbk[x^{\pm 1}]^{\circ}$ consists of those sequences that extend across all of $\bZ$, satisfying the linear recurrence throughout. Per \Cref{pr:hcffdalg}, that Hopf subalgebra is in fact the {\it universal} bialgebra morphism received by $\Bbbk[x]^{\circ}$ from a Hopf algebra, so is in particular also the {\it largest} Hopf subalgebra of the bialgebra of linearly recursive sequences. 
\end{enumerate}
\end{remarks}

A small subtlety present in the proof of \Cref{th:extprescf} (specifically, the implication \Cref{item:th:extprescf:alg} $\Rightarrow$ \Cref{item:th:extprescf:alg}) suggests a side-note:
\begin{itemize}[wide]
\item On the one hand, \Cref{le:presfindual} handles, for algebraic $\Bbbk\le \Bbbk'$, preservation by $(-)_{\Bbbk'}$ of finite duals for {\it affine} algebras.

\item On the other hand, affineness cannot simply be dropped in general (\Cref{ex:infncpoly}).

\item The issue of applying \Cref{le:presfindual} in the course of the proof of the aforementioned implication suggests itself, and with it so does the question of when the free Hopf algebra $H_{\cat{f}}(C)$ on a finite-dimensional coalgebra is affine. Such $C$, it turns out, can be classified. 
\end{itemize}

\begin{proposition}\label{pr:whenhfaff}
  The free Hopf algebra $H_{\cat{f}}(C)$ on a finite-dimensional $\Bbbk$-coalgebra $C$ is $\Bbbk$-affine if and only if the scalar extension $C_{\overline{\Bbbk}}$ to the algebraic closure is pointed. 
\end{proposition}
\begin{proof}
  As a consequence \cite[\S 3.5, Corollary b]{pierce_assoc} of the {\it Wedderburn-Artin} theorem \cite[\S 3.5, Theorem]{pierce_assoc}, the coradical $(C_{\overline{\Bbbk}})_0$, being finite-dimensional and cosemisimple over an algebraically closed field, must be a finite sum
  \begin{equation*}
    \left(C_{\overline{\Bbbk}}\right)_0
    \cong
    \bigoplus_{i=1}^{\ell}
    M_{n_i}(\overline{\Bbbk})^*
  \end{equation*}
  of matrix coalgebras; the analogous isomorphism then holds for sufficiently large (in the sense of inclusion) {\it finite} $\Bbbk''\ge \Bbbk$. Affineness being preserved and reflected by finite field extensions we can, upon substituting for $\Bbbk'$ a {\it composite field} \cite[\S 11.7]{cohn_alg} of $\Bbbk',\Bbbk''\ge \Bbbk$ (thus finite over $\Bbbk'$), assume that $C_0\cong \bigoplus M_{n_i}^*$ to begin with. Under that assumption, the claim is that $H_{\cat{f}}(C)$ is affine precisely when all $n_i$ are 1.

  Choosing a (plain, vector-space) decomposition $C\cong C_0\oplus V$, \cite[Theorem 32]{zbMATH03344702} gives the algebra structure
  \begin{equation*}
    H_{\cat{f}}(C)
    \cong
    T(V)\sqcup \coprod_{i}H_{\cat{f}}(M_{n_i}^*)
    \quad
    (\text{coproduct in }\cat{Alg}). 
  \end{equation*}
  Affineness for $H_{\cat{f}}(C)$ as a whole thus reduces to affineness for each individual $H_{\cat{f}}(M_{n_i}^*)$, and hence the desired equivalence again reduces:
  \begin{equation*}
    H:=H_{\cat{f}}(M_n^*)\text{ is affine}
    \xLeftrightarrow{\quad\text{claim}\quad}
    n=1.
  \end{equation*}
  One implication ($\Leftarrow$) is plain enough: $H_{\cat{f}}(\Bbbk)\cong \Bbbk[x^{\pm 1}]^{\circ}$ \cite[Lemma 34]{zbMATH03344702}, certainly affine. For the converse, the basis for $H$ described in \cite[Theorem 5]{nic} makes it clear that while generated as an algebra by the iterates
  \begin{equation*}
    (S^\bullet M_n^*)_{\bullet\in \bZ_{\ge 0}}
    ,\quad
    S:=\text{antipode of }H,
  \end{equation*}
  $H$ is {\it not} generated as an algebra by only finitely many of those iterates. 
\end{proof}

\begin{remark}\label{re:cpointedbijant}
  The coalgebras featuring in \Cref{pr:whenhfaff}, with pointed extensions to the algebraic closure, are also precisely those whose free Hopf algebra happens to have bijective antipode \cite[Theorem 18]{zbMATH03344702}. 
\end{remark}

\addcontentsline{toc}{section}{References}

\Addresses

\end{document}